\documentclass[11pt]{amsart}

\usepackage{xcolor} 
\usepackage{fancybox}
\usepackage{dashbox}

  \usepackage{float}

  \usepackage{pgffor}
  \usepackage{tikz}
  \usetikzlibrary{calc,3d,arrows,positioning,shapes.misc}

  \usepackage[english]{babel}
  \usepackage{amsthm}
  \usepackage{amsmath}
  \usepackage{amssymb}
  \usepackage[latin1]{inputenc}
  \usepackage{setspace}
  \usepackage{enumerate}
  \usepackage{stmaryrd}

  \usepackage[colorlinks=true,linkcolor=black,citecolor=black,filecolor=black,urlcolor=black,menucolor=black]{hyperref}

  \usepackage{array}
  \usepackage{graphicx}
  \usepackage[babel]{csquotes}
  \usepackage{geometry}
  \usepackage{bbm}
  \usepackage[all]{xy}
  \usepackage{mathrsfs}

\usepackage{mathtools}

  \theoremstyle{plain}
  \newtheorem{theorem}{Theorem}
  \newtheorem*{theorem*}{Theorem}
  \newtheorem{proposition}{Proposition}
  \newtheorem{corollary}{Corollary}
  
  \newtheorem{lemma}{Lemma}
  
  \theoremstyle{definition}
  \newtheorem{definition}{Definition}

  \newtheorem{remark}{Remark}
  \newtheorem*{notation*}{Notation}

  \newcommand{\step}[1]{\par\medskip\par\noindent\textit{#1}} 
	\allowdisplaybreaks[3]

  \def \H{\mathcal{H}}

  \def \R{\mathbb{R}}
  
	\def \S{\mathbb{S}}
   
  \newcommand {\trace} {\mathop \textup{tr}}

  \newcommand {\supp} {\mathop \mathrm{supp}}

  \newcommand {\eps} {\varepsilon}

  \newcommand {\uu} {\mathbf{u}}
  \newcommand {\ee} {\mathbf{e}}
  \newcommand {\ff} {\mathbf{f}}
  \newcommand {\torus} {\mathbb{T}^d}
  \newcommand {\cdotdot} {\mathbin{:}}
 
\newcommand{\dd}{\mathrm{d}}

\usepackage[
backend=biber,
style=numeric-comp,
giveninits=true,
url=false,
maxbibnames=4
]{biblatex}
\AtEveryBibitem{\clearfield{issn}}
\AtEveryBibitem{\clearfield{isbn}}
\renewbibmacro{in:}{%
  \ifentrytype{article}
    {}
    {\printtext{\bibstring{in}\intitlepunct}}}

  \title[Convergence of Allen--Cahn system to multiphase mean curvature flow]{Convergence of a vector-valued Allen--Cahn system to Brakke's multiphase mean curvature flow}

\author{Tim Laux}
\address{(Tim Laux) Institute for Mathematics and Interdisciplinary Center for Scientific Computing\\
         Heidelberg University\\
         Im Neuenheimer Feld 205\\
         D-69120 Heidelberg\\
         Germany}
\email{tim.laux@math.uni-heidelberg.de}

\author{Keisuke Takasao}
\address{(Keisuke Takasao) Department  of Mathematics\\
Graduate School of Science\\
        Kyoto University\\
        Kitashirakawa-Oiwakecho Sakyo\\
        Kyoto 606-8502\\
        Japan}
\email{k.takasao@math.kyoto-u.ac.jp}

    \date{\today}

\begin{document}

    \begin{abstract}
      Since Ilmanen's pioneering work~[J. Differential Geom. 38, 417-461, (1993)] it has been a major open problem to understand the sharp-interface limit of systems of coupled Allen--Cahn equations.
      We prove---for the first time---a global, unconditional convergence result for such a coupled system and show that the limit is a multiphase mean curvature flow in the sense of Brakke. 

      \medskip


  \medskip

\noindent \textbf{Mathematical Subject Classification (MSC2020)}: 
53E10 (primary); 
49Q20; 
35A15; 
35B25; 
35D30 (secondary) 
  \end{abstract}
\maketitle

\tableofcontents

\section{Introduction}
We study the singular limit $\eps\to0$ of solutions $\uu_\eps$ to the vector-valued Allen--Cahn equation
\begin{align}\label{eq:allen cahn loosely}
  \eps \partial_t \uu_\eps = \eps \Delta \uu_\eps - \partial_{\uu} F_\eps(\uu_\eps).
\end{align}
Here, $F_\eps $ is a multi-well potential that in the limit $\eps\to0$ forces $\uu_\eps$ to take on only finitely many values. 
This is the $L^2$-gradient flow (on a slow time scale $\sim \eps$) of the Ginzburg--Landau (or Cahn--Hilliard) energy
\begin{align}\label{eq:Eeps loosely}
  E_\eps(\uu) 
  &= 
  \int \Big(\frac{\eps}2 |\nabla \uu|^2 + F_\eps (\uu) \Big)\, \dd x.
\end{align}
The present work proves for the first time in the vector-valued setting, that---globally in time and unconditionally---the solution of our Allen--Cahn system~\eqref{eq:allen cahn loosely} converges to a multiphase mean curvature flow. 
This draws a connection between a system of parabolic partial differential equations and one of the most studied geometric flows.

\medskip

More precisely, the solutions are pre-compact and any limit $\uu=\lim_{\eps\to0} \uu_\eps$ only takes values in the energy wells of $F_\eps$, therefore providing a natural partition $(\Omega_1,\ldots,\Omega_n)$ of the spatial domain, cf.\ Figure~\ref{fig:partition}. 
Second, the network of interfaces $\Sigma_{ij}=\partial \Omega_i \cap \partial \Omega_j$ $(1\leq i,j\leq n)$ that separate these domains moves by multiphase mean curvature flow. 
Loosely speaking, this means that the normal velocity is equal to the negative mean curvature, i.e., 
\begin{equation}\label{eq:MCF}
   V_{ij} = - H_{ij}\quad \text{on each interface } \Sigma_{ij},
\end{equation}
and triple junctions are in local equilibrium, i.e.,
\begin{equation}\label{eq:Herring}
  \nu_{ij}+\nu_{jk}+\nu_{ki} =0 \quad \text{along each triple junction } \Sigma_{ij}\cap \Sigma_{jk}\cap \Sigma_{ki}.
\end{equation}
\begin{figure}
  \includegraphics[width=0.4\linewidth]{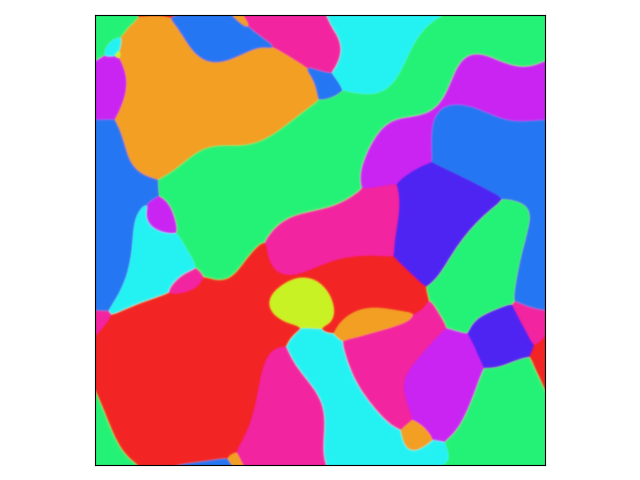}%
  \includegraphics[width=0.4\linewidth]{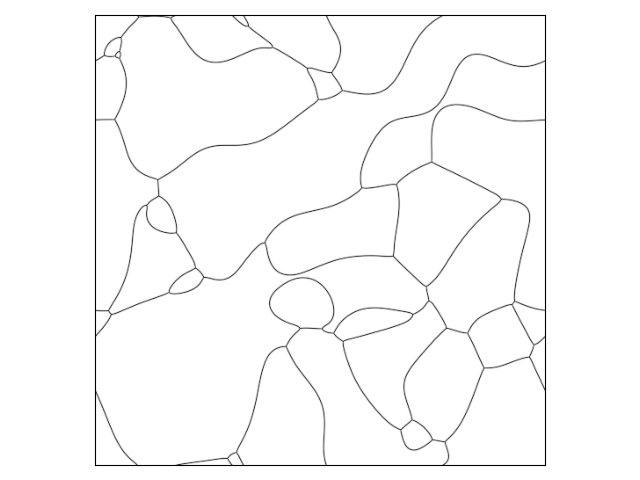}
  \label{fig:partition}
\caption{A partition: Indication of the phases, different values of $\uu$ correspond to different colors (left); the corresponding network of interfaces (right).}
\end{figure}%
Here, $V_{ij}$, $H_{ij}$, and $\nu_{ij}$ denote the velocity, mean curvature and unit normal field of the interface $\Sigma_{ij}$. 
Classical solutions to the curve shortening flow with junctions have been extensively studied over the years, starting with the works~\cite{BronsardReitich,KinderlehrerLiu,MantegazzaNovagaTortorelli}; see~\cite{IlmanenNevesSchulze,MantegazzaNovagaPludaSchulze} and the references therein.
However, since multiphase mean curvature flow generically undergoes topological changes and develops singularities in finite time, we need to resort to a weak formulation of~\eqref{eq:MCF}--\eqref{eq:Herring} in the context of geometric measure theory. 
Our main result, Theorem~\ref{thm:ACtoBrakke}, can be informally stated as follows.
\begin{theorem*}
  For well-prepared initial data and suitable multi-well potentials, the solutions of the Allen--Cahn equation~\eqref{eq:allen cahn loosely} converge to a weak solution of multiphase mean curvature flow~\eqref{eq:MCF}--\eqref{eq:Herring} in the sense of Brakke.
\end{theorem*}

Loosely speaking, Brakke's formulation of multiphase mean curvature flow encodes~\eqref{eq:MCF}--\eqref{eq:Herring} in a family of energy-dissipation relations of the form
\begin{equation}\label{eq:Brakke_loosely}
  \frac{\dd}{\dd t}  \sum_{i\neq j}\int_{\Sigma_{ij}(t)} \varphi(\cdot,t) \,\dd S
  \leq   \sum_{i\neq j}\int_{\Sigma_{ij}(t)} 
  \Big( 
    -\varphi H_{ij}^2 +  H_{ij} \nu_{ij} \cdot \nabla \varphi + \partial_t \varphi
  \Big)
  \, \dd S
\end{equation}
for all non-negative, sufficiently smooth test functions $\varphi\geq0$.

\medskip
Such a convergence result is highly challenging for several reasons. 
While multiphase mean curvature flow is formally a gradient flow, namely of the total interface area functional
\[
  E(\uu ) =  \sum_{i\neq j} \mathrm{Area}(\Sigma_{ij}),
\]
the corresponding metric in the gradient-flow structure, the $L^2$-metric on normal velocities, is completely degenerate~\cite{MichorMumford}. 
Moreover, in contrast to the scalar case, the geometry is much more complex in the multiphase case, and no maximum principle is available.

\medskip

Ever since Ilmanen's pioneering work~\cite{Ilmanen_allencahn}, it has been a major open question to prove convergence for any type of Allen--Cahn system to multiphase mean curvature flow, cf.~\cite[Section 13, Question 5]{Ilmanen_allencahn}.
There are three partial results in this direction in the literature: 
Bronsard and Reitich~\cite{BronsardReitich} provided formal asymptotic expansions near a triple junction that suggest this limiting behavior. 
In a joint work with Simon~\cite{LauxSimon_allencahn}, the first author proved a \emph{conditional} convergence result towards~$BV$ solutions for a large class of potentials~$F_\eps$, which hinges on the assumption that the energy converges in the limit~$\eps\to0$.
More recently, building on the scalar result~\cite{FLS_allencahn}, Fischer and Marveggio~\cite{FischerMarveggio_vectorAC_special3wells} proved the convergence for short time, more precisely, under the condition that the limiting mean curvature flow remains smooth, and for specific, carefully constructed potentials. 
In contrast, the result of the present paper is unconditional, and global in time.
The only restriction---but also the key idea of the present work---is the design of the potential $F_\eps$, whose construction is elementary, even for any number of phases, but has a special structure. 
The main point is that the coupling between the different components is slightly weaker than in~\cite{LauxSimon_allencahn}. 
This allows us to use integral estimates by X.\ Chen~\cite{Chen_global} to prove the vanishing of the discrepancy measure in the limit.
Moreover, we can even get stronger, pointwise bounds on the discrepancy measure, based on prior work of Tonegawa et al.~\cite{HutchinsonTonegawa,MizunoTonegawa,TakasaoTonegawa_transport,RoegerTonegawa}, which in turn allows us to prove a monotonicity formula in the spirit of Ilmanen~\cite{Ilmanen_allencahn}.

\medskip

The singular limit of the Allen--Cahn equation has a long history in the mathematical literature.
In the static case, the convergence of minimizers has been answered by Modica and Mortola~\cite{ModicaMortola,Modica}. Baldo~\cite{Baldo} generalized this to the setting of a broad class of multi-well potentials. 
The asymptotics of critical points of the Cahn--Hilliard energy, i.e., steady-state solutions of~\eqref{eq:allen cahn loosely}, are more delicate and have only been addressed in the scalar case. 
Hutchinson and Tonegawa~\cite{HutchinsonTonegawa} proved the convergence of critical points to generalized minimal surfaces, which are integral varifolds with vanishing first variation. 
If the energy converges in the limit, the earlier work of Luckhaus and Modica~\cite{LuckhausModica} shows that the first variation of the energies converge, in particular showing that the limiting minimal surface has unit multiplicity.
This conditional result also generalizes to the vector-valued case, as Simon and one of the authors~\cite[Proposition~3.1]{LauxSimon_allencahn} showed, but also in this static case, no unconditional vectorial convergence result is known.

\medskip

In our time-dependent case~\eqref{eq:allen cahn loosely}, there are two main avenues: either proving global-in-time convergence results using compactness, or local-in-time results for smooth initial data by exploiting the smoothness of mean curvature flow. 
For short time, before the onset of singularities, there are two methods available.
X.\ Chen~\cite{chen_smooth} and De~Mottoni and Schatzman~\cite{deMottoniSchatzmann} independently proved quantitative convergence of the scalar Allen--Cahn equation to two-phase mean curvature flow by controlling the error between the true solution of the Allen--Cahn equation to a smooth profile around the limiting mean curvature flow.
This, however, requires an intricate spectral gap estimate which has not been generalized to the vectorial multiphase setting, see~\cite{AbelsMoser_vectorAC_2wells,Linetal23KellerRubinsteinSternberg} for vectorial versions with only two wells.
The other, more recent approach is the relative energy method for phase-field models by Fischer, Simon and one of the authors~\cite{FLS_allencahn}, which is more adapted to the geometric nonlinearity of (multiphase) mean curvature flow, see also~\cite{FHLS} for the general structure in the multiphase case.
This method has proven to be versatile and has since been applied to several variants of the Allen--Cahn equation, among which a mass-preserving version that approximates volume-preserving mean curvature flow~\cite{KroemerLaux_nonlocalAC} and the above-mentioned work~\cite{FischerMarveggio_vectorAC_special3wells}, which also uses the idea of gradient-flow calibrations from~\cite{FHLS}.

\medskip

In order to prove convergence on long time scales also past singularities, weak notions have to be employed. 
Evans, Soner and Souganidis~\cite{EvansSonerSouganidis} proved the convergene of the scalar Allen--Cahn equation to the viscosity solution of two-phase mean curvature flow via the comparison principle. 
However, in the multiphase case, these comparison techniques do not apply.
Instead, one should resort to the gradient-flow structures of the Allen--Cahn equation and of multiphase mean curvature flow.
This idea was first pointed out by Bronsard and Kohn~\cite{BronsardKohn}, who also proved convergence in the scalar, radially symmetric case. 
Also Brakke's weak formulation~\eqref{eq:Brakke_loosely} is guided by the gradient-flow structure. 
Ilmanen~\cite{Ilmanen_allencahn} proved the convergence of the scalar Allen--Cahn equation to a two-phase Brakke flow.
The key of his work is a monotonicity formula that goes beyond the simple monotonicity of the total energy
\begin{equation}\label{eq:EDI_loosely}
  \frac{\dd}{\dd t} E_\eps(\uu_\eps(\cdot,t)) 
  = -\int \frac1\eps \big|\eps \partial_t \uu_\eps \big|^2\,\dd x,
\end{equation}
and instead states that the energy on small scales at later times can be controlled by the energy on larger scales at earlier times up to an error on which we have precise control. 
This precisely mirrors Huisken's celebrated monotonicity formula for mean curvature flow~\cite{Huisken_monotonicity}. 
Such monotonicity formulas have an intimate connection to regularity theory.
Ilmanen's work has been generalized to several other cases: One of the authors showed the convergence in the volume-preserving case~\cite{Takasao_volpres,Takasao_vol_preshigher_dim}, in the presence of obstacles~\cite{Takasao_obstacle}, and---together with Tonegawa---included a given advective field~\cite{TakasaoTonegawa_transport}; Hensel and the other author~\cite{HenselLaux_degiorgi} showed that Ilmanen's limit is also a De~Giorgi solution and therefore also enjoys a weak-strong uniqueness principle. 
However, this method did not seem to extend to the vectorial case.
The only long-time convergence results in the vectorial setting hinge on the assumption that the energies converge in the limit: Simon and one of the authors~\cite{LauxSimon_allencahn} showed the convergence to a $BV$ solution, the other author showed convergence to a Brakke flow~\cite{takasao2018LL}, and Steinke~\cite{Steinke} proved convergence to a De~Giorgi solution.

\medskip

The main difficulty in the vectorial Allen--Cahn equation is that such a monotonicity formula must depend on the potential $F_\eps$ in intricate ways: Many other physical systems have the same formal structure as the vectorial Allen--Cahn equation~\eqref{eq:allen cahn loosely}, but do not exhibit such a monotonicity formula. 
For example, the complex-valued Ginzburg--Landau equation corresponds to~\eqref{eq:allen cahn} with $n=2$, with the crucial difference that the potential $F_\eps (\uu) = \frac1\eps (1-|\uu|^2)^2$ vanishes along a circle. 
If $d\geq3$, the singular set of the limit $\uu$ moves by (co-dimension $2$) mean curvature flow, see e.g.~\cite{Lin98,JerrardSoner98}, but the monotonicity formula of Ilmanen does not hold in this case.
More generally, one can consider potentials that vanish on a submanifold of the target space, or a union of submanifolds, cf.~\cite{LinWang12}. 
One of these examples is the Keller--Rubinstein--Sternberg model, for which $\{F_\eps=0\}= O(d) \subset \R^{d\times d}$ is continuous but has two connected components. Fei, Lin, Wang, and Zhang~\cite{Linetal23KellerRubinsteinSternberg} were able to carry out the program of error bounds for matched asymptotic expansions~\cite{chen_smooth,deMottoniSchatzmann} for this model, and proved that for short time, solutions converge to two-phase mean curvature flow.
Another interesting example comes from liquid crystals at the critical temperature in which the wells of $F_\eps$ are $\{0\} \cup \mathcal{U}$ where $\{0\}$ is the isotropic phase and $\mathcal U$ is the nematic phase, a three-dimensional submanifold with a conical singularity. Liu and one of the authors proved the convergence towards two-phase mean curvature flow for short time~\cite{LauxLiu21} via the method from~\cite{FLS_allencahn}.
Global-in-time convergence results for such systems are widely open, even in the case when $\{F_\eps=0\}$ only has two connected components.
A key obstacle in proving long-time results for systems like~\eqref{eq:allen cahn loosely} is the crucial difference between the cases when the wells $\{F_\eps=0\}$ are discrete or continuous. 
In our case of discrete wells, it has been expected that, like in the scalar case, the energy is equally distributed between two terms~$\frac\eps2 |\nabla \uu_\eps|^2$ and~$F_\eps(\uu_\eps)$, which we will show then implies a monotonicity formula. 
This is in stark contrast to the case of continuous wells, where of course, the field can freely move within the well and, for example in the case of the complex-valued Ginzburg--Landau equation, it is known that the energy is concentrated in the first term $\frac\eps2 |\nabla \uu_\eps|^2$. 
In that case, one can only prove a monotonicity formula with a kernel that rescales with co-dimension $2$, cf.~\cite{BethuelOrlandiSmets}, which is natural as the limiting singular sets are co-dimension $2$ surfaces. But in our case we need the sharper scaling of co-dimension $1$ as we expect to see networks of hypersurfaces in the limit.
A related question is the structure of solutions of vector-valued Allen--Cahn equations close to triple junctions. 
Bronsard, Gui, and Schatzman raised this question in~\cite{BronsardGuiSchatzman}, and it has received continuous attention since then with significant contributions by Alikakos starting with~\cite{Alikakos11}; see the work by Alikakos, Geng, and Zarnescu~\cite{Alikakosetal22} as well as Sandier and Sternberg~\cite{SandierSternberg24}, and references therein.
In Alikakos' phrasing, our monotonicity formula corresponds to the long sought-after ``strong monotonicity formula'', while for other systems, only the weak monotonicity formula (like the one for the complex-valued Ginzburg--Landau equation) has been known. 

\medskip

Brakke flows are well-studied objects in geometric measure theory.  
We refer the interested reader to the lecture notes of Tonegawa~\cite{Tonegawa_SpringerBriefs}. 
In particular, for sufficiently regular evolving surfaces, the localized energy-dissipation inequality~\eqref{eq:Brakke_loosely} contains sufficient information to characterize them as solutions to~\eqref{eq:MCF}--\eqref{eq:Herring}, see~\cite[Proposition~2.1]{Tonegawa_SpringerBriefs}.
Moreover, under suitable additional conditions, Brakke flows are amenable to a partial regularity theory, first established by Brakke~\cite{brakke} (see also \cite[Theorem~6.1]{Tonegawa_SpringerBriefs}), and subsequently studied by White~\cite{White2005}, Kasai and Tonegawa~\cite{Kasai-Tonegawa,Tonegawa2014}, and Stuvard and Tonegawa~\cite{stuvard2025epsilonregularitytheorembrakkeflows}. 
These regularity theorems must also account for singularities such as junctions. 
The precise structure of these singularities was analyzed by Cheeger, Haslhofer, and Naber~\cite{Cheeger-Haslhofer-Naber2013}.
It should be noted that these regularity results require the Brakke flow to have unit density. 
Whether the solutions obtained in this paper possess unit density remains a subject for future research.

\medskip

Our convergence result has ramifications in other areas in- and outside of mathematics.
First, to simulate multiphase mean curvature flow~\eqref{eq:MCF}--\eqref{eq:Herring} through singularities and topological changes, one only needs to solve the Allen--Cahn equation~\eqref{eq:allen cahn loosely}, which is a system of parabolic differential equations. 
Second, the limit $\eps\to0$ justifies an effective field theory.
It connects the model~\eqref{eq:allen cahn loosely} which is part of Ginzburg--Landau theory, which has a long history in physics, dating back to the work of Landau~\cite{Landau36} and Ginzburg--Landau~\cite{GinzburgLandau}, to a purely geometric flow.
Third, while our result concerns the sharp-interface limit of the Allen--Cahn equation, one may also read it as an existence result for multiphase mean curvature flow, which is an interesting problem in geometric measure theory in its own right.
For two-phase problems, weak solutions can be constructed by various methods. 
Notable examples are the construction of Brakke flows via elliptic regularization by Ilmanen~\cite{Ilmanen1994}, and the construction of viscosity solutions using the level-set method by Chen, Giga, and Goto \cite{ChenGigaGoto} and Evans and Spruck \cite{EvansSpruckI}. 
In particular, Evans and Spruck also showed that these viscosity solutions form Brakke flows~\cite{EvansSpruckIV}. 
Vice versa, Ullrich and one of the authors~\cite{LauxUllrich} showed that generically, two-phase Brakke flows are in fact equivalent to the viscosity solution, as long as their enclosed volume is continuous in time. 
In the multiphase case, Kim, Stuvard, and Tonegawa~\cite{KimTonegawa,StuvardTonegawa} proved the global existence of a multiphase Brakke flow by establishing a time-discrete variational approximation scheme that requires two procedures per time step, originally proposed by Brakke.
While this procedure provides a solid foundation, adapting it to other situations requires substantial technical effort, see e.g.~\cite{scapin2026} for the adaptation of the method from the whole space to a half-space with homogeneous Neumann conditions.
In contrast, our method offers the advantage of avoiding complex procedures, as the proof relies fundamentally on the properties of solutions to the system of reaction-diffusion equations.
We are convinced that our method will shed light on several related problems as well.

\section{Main results}

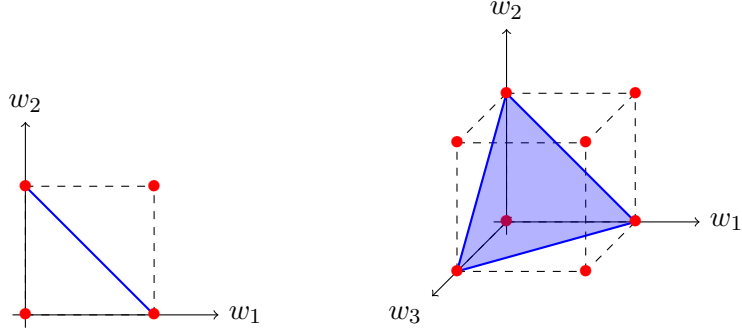
\begin{figure}
  \begin{tikzpicture}[scale=1.7]
    \draw[->] (-0.1,0) -- (1.5,0) node [right]{$w_1$};
    \draw[->] (0,-0.1) -- (0,1.5) node [above]{$w_2$};
    
    \foreach \i in {0,1}
    {
          \draw[dashed] (0,\i) -- (1,\i);
          \draw[dashed] (\i,0) -- (\i,1);
    }

    \draw[thick,blue] (1,0) -- (0,1);

    \foreach \i in {0,1}
    {
      \foreach \j in {0,1}
      {
        \node[red] at (\i,\j) {\textbullet};
      }
    }

  \end{tikzpicture}
  $\qquad\quad$
  \begin{tikzpicture}[scale=1.7]
    \draw[->] (-0.1,0,0) -- (1.5,0,0) node [right]{$w_1$};
    \draw[->] (0,-0.1,0) -- (0,1.5,0) node [above]{$w_2$};
    \draw[->] (0,0,-0.1) -- (0,0,1.5) node [below left]{$w_3$};

    \foreach \i in {0,1}
    {
      \foreach \j in {0,1}
      {
          \draw[dashed] (0,\i,\j) -- (1,\i,\j);
          \draw[dashed] (\i,0,\j) -- (\i,1,\j);
          \draw[dashed] (\i,\j,0) -- (\i,\j,1);
      }
    }

    \node[red] at (0,0,0) {\textbullet};
    \draw[thick,blue] (1,0,0) -- (0,1,0) -- (0,0,1) -- cycle;
    \draw[fill,blue,opacity=0.3] (1,0,0) -- (0,1,0) -- (0,0,1) -- cycle;

    \foreach \i in {0,1}
    {
      \foreach \j in {0,1}
      {
        \foreach \k in {0,1}
        {
          \ifthenelse{\i=0 \AND \j =0 \AND \k=0}{}{
            \node[red] at (\i,\j,\k) {\textbullet};
          }
        }
      }
    }

  \end{tikzpicture}
  
  \caption{Illustrations of the energy landscape for $n=2$ (left) and $n=3$ (right) in the transformed coordinates $w_i=\frac1\sigma\phi(u_i)$. 
  The wells of $\frac1\eps \sum_{i=1}^n W(\phi^{-1}(\sigma w_i))$ (red points) are the $2^n$ vertices of the unit $n$-cube $(0,1)^n$. 
  The well of $\frac{\sigma^2}{2\eps^\alpha} \big( 1 - \sum_{i=1}^n w_i\big)^2$ (blue surface) is the standard $(n-1)$-simplex $\mathrm{conv}\{\ee_1,\ldots,\ee_n\}$.
  The intersection of the two is precisely the desired pure phases $\{\ee_1,\ldots\ee_n\}$.}
  \label{fig:potential}
\end{figure}
We first motivate our potential. 
Let $W(u)=u^2 (1-u)^2$, $u\in \R$, denote the standard double-well potential, $\phi(u)= \int_0^u \sqrt{2W(s)}\,\dd s$ and $\sigma= \phi(1)$.
We can think of $u\mapsto w$ with $w_i=\frac1\sigma \phi(u_i)$ as a coordinate transformation in our state space.
Then $w = 1$ if $u= 1$ and $w= 0$ for $u= 0$. 
We want to define a potential that enforces the constraint $\uu \in \{\ee_1,\dots,\ee_n\}$ in the limit $\eps\to0$, but that keeps the coupling between the different components of $\uu_\eps$ sufficiently weak. 
For $\uu = (u_1,\ldots,u_n)\in \R^n$, we therefore set
\begin{equation}\label{eq:potential}
  F_\eps(\uu):= \frac1\eps \sum_{i=1}^n W(u_i) + \frac1{2\eps^\alpha} \Big( \sigma - \sum_{i=1}^n \phi (u_i)\Big)^2
  \quad \text{with } \alpha\in (0,1),
\end{equation}
see Figure~\ref{fig:potential}.
In terms of the change of variables above, the second term then simply reads $\frac{\sigma^2}{2\eps^\alpha} \big(1-\sum_{i=1}^n w_i\big)^2$ and couples the different components by weakly enforcing the sum constraint $\sum_{i=1}^n w_i =1$.

Then the energy $E_\eps$ from~\eqref{eq:Eeps loosely} of a vector-valued function $\uu = (u_1,\ldots,u_n)\colon \torus \to \R^n$ reads
\begin{align}
  \label{eq:Eeps}
  E_\eps(\uu)=
 \sum_{i=1}^n \int_{\torus} \frac{\eps}2 |\nabla u_i|^2 + \frac1\eps W(u_i)  \, \dd x + \int_{\torus} \frac1{2\eps^\alpha}\Big( \sigma - \sum_{i=1}^n \phi ( u_i)\Big)^2 \, \dd x.
\end{align}
Here and throughout, we restrict ourselves to the case of periodic functions, i.e., functions defined on the flat torus $\torus$.
The first integral is the usual scalar Allen--Cahn energy for each component of $\uu$, while the second integral is a penalty term for violating the condition $\sum_{i=1}^n w_i=1$. 
Hence, in the limit $\varepsilon \to 0$, the constraint $\sum_{i=1}^n w_i=1$ is enforced, together with the usual constraints  $u_i\in \{0,1\}$ (coming from the penalty $\frac1\eps W(u_i)$).
Hence, we expect that any $\uu_\eps$ with $\sup_{\eps\in(0,1]} E_\eps (\uu_\eps) <\infty$ and with $\uu_\eps \to \uu$ in $L^1$ yields a partition $\uu = \sum_{i=1}^n \chi_{\Omega_i} \ee_i$ in the limit.
One can easily show (see Proposition~\ref{prop:Gamma_convergence}) that $E_\eps$ indeed $\Gamma$-converges to the optimal partition energy
\begin{align}\label{eq:E}
  E(\uu) = 
    \sigma \sum_{i=1}^n P(\Omega_i) \quad \text{for} \quad \uu = \sum_{i=1}^n \chi_{\Omega_i} \ee_i \text{ with } (\Omega_i)_{i=1}^n\text{ a partition of } \torus.
\end{align}
Here $P(\Omega):= \sup\{\int_{\Omega} \nabla \cdot X \,\dd x \colon X\in C^1(\torus;\R^d), |X|\leq 1 \text{ in }\torus\}$ denotes the perimeter of a measurable set $\Omega\subset \torus$. 
This means in particular that while the penalty term enforces the sum constraint $\sum_{i=1}^n w_i=1$ in the limit, it does not change the value of the energy functional at finite energy configurations.

\medskip

In spirit, the effect of the potential on our dynamics is reminiscent of the Voronoi implicit interface method by Saye and Sethian~\cite{SayeSethian}, an efficient and versatile time-dicretization of multiphase mean curvature flow based on the level set method of Osher and Sethian~\cite{OsherSethian}.
This method alternates between two steps: first, individual mean curvature flow for each set in the partition; second, filling the voids around triple junctions by geometric optics.
In our case, while the gradient flow of the first integral in the energy~\eqref{eq:Eeps} would evolve each phase by a scalar Allen--Cahn equation, the second integral in~\eqref{eq:Eeps} couples the equations to fill the void.

\medskip

We write equation~\eqref{eq:allen cahn loosely} explicitly as the following system, for $1\leq i \leq n$,
\begin{align}
  \label{eq:allen cahn}
  \eps\partial_t u_i 
  &= \eps\Delta u_i -\frac1\eps W'(u_i) + \frac{1}{\eps^{\alpha}}\sqrt{2W(u_i)} \Big(\sigma - \sum_{j=1}^n \phi(u_j)\Big), 
  &&\text{in } \torus \times (0,\infty),
  \\
  u_i(\cdot,0) 
  &= u_{0,i} 
  &&\text{on } \torus \times\{t=0\},
\end{align}
with a given initial datum $\uu_0=(u_{0,i})_{i=1,\cdots,n}$, see Figure~\ref{fig:allen-cahn} for an example.

\begin{figure}
  \includegraphics[width=0.2\textwidth]{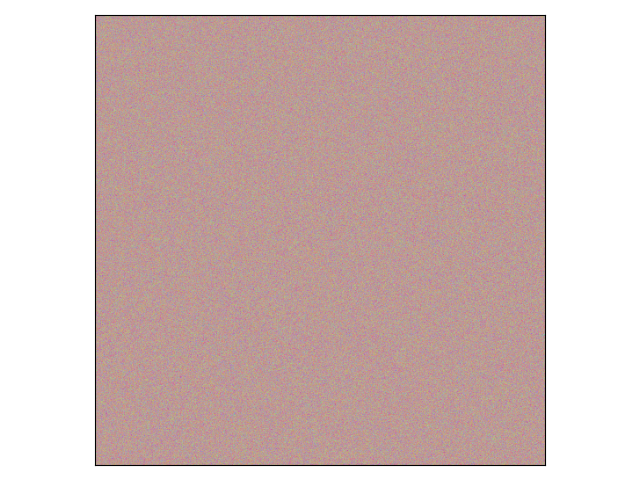}%
  \includegraphics[width=0.2\textwidth]{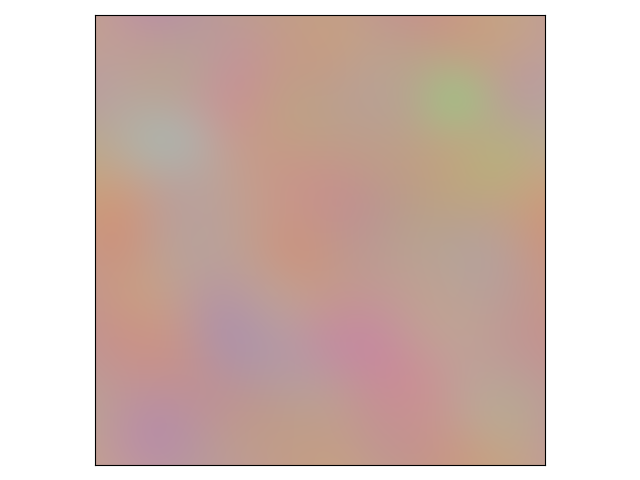}%
  \includegraphics[width=0.2\textwidth]{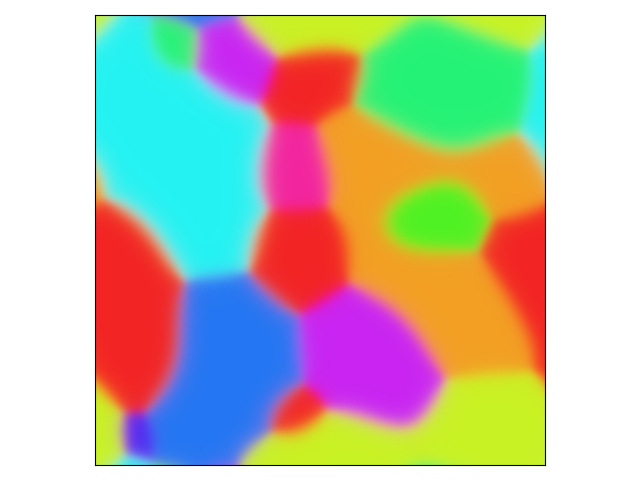}%
  \includegraphics[width=0.2\textwidth]{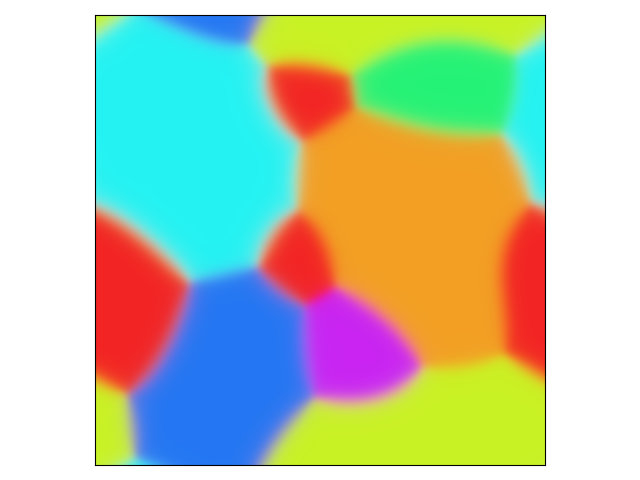}%
  \includegraphics[width=0.2\textwidth]{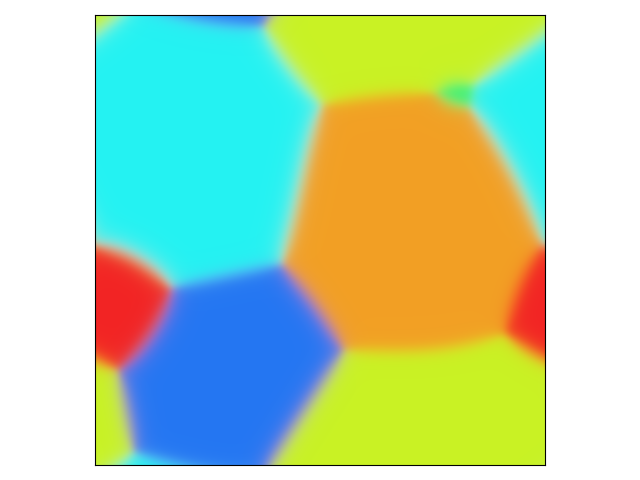}%

  \vspace{10pt}
  \includegraphics[width=0.2\textwidth]{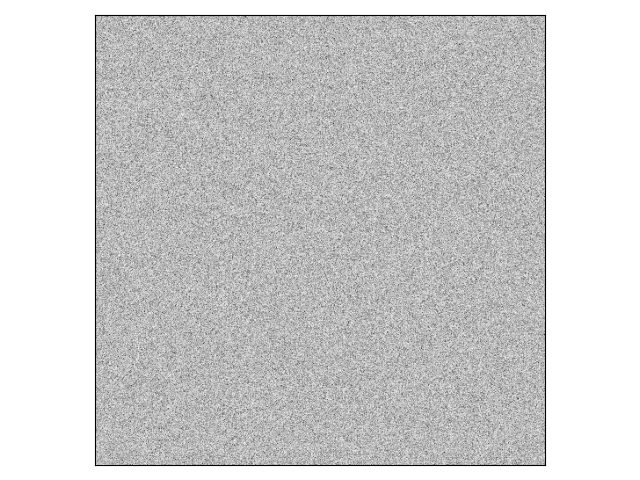}%
  \includegraphics[width=0.2\textwidth]{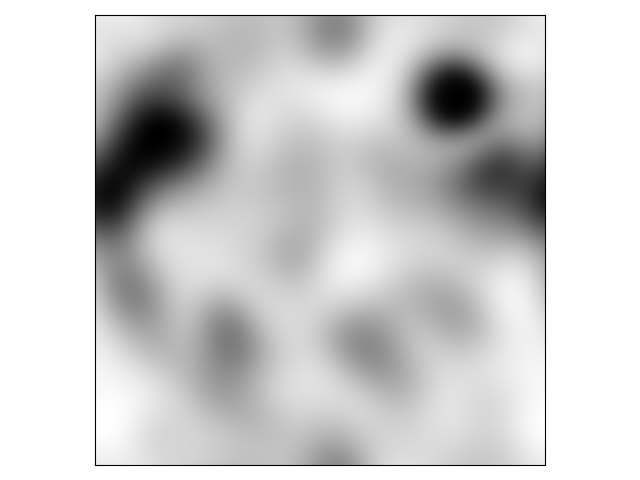}%
  \includegraphics[width=0.2\textwidth]{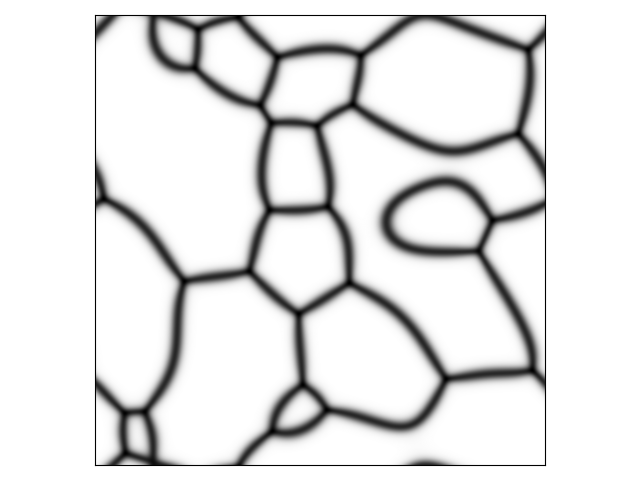}%
  \includegraphics[width=0.2\textwidth]{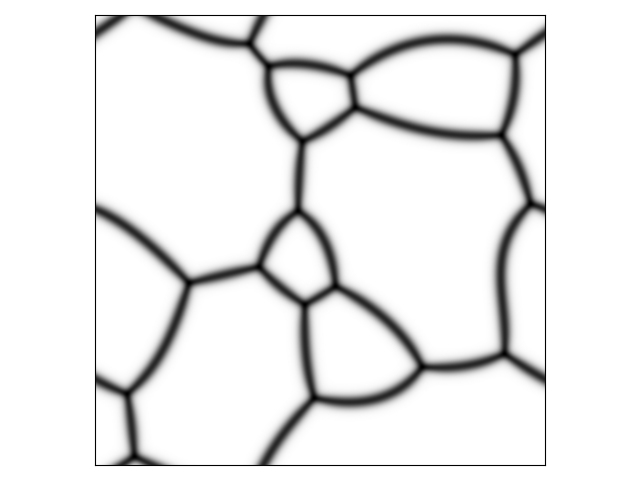}%
  \includegraphics[width=0.2\textwidth]{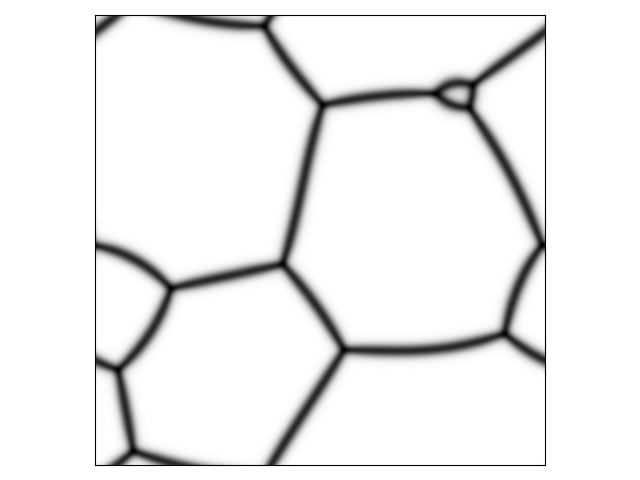}%
  \caption{The solution to the Allen--Cahn system~\eqref{eq:allen cahn} starting from random initial data in dimension $d=2$. 
  After a short time, the data becomes well-prepared, then our result, Theorem~\ref{thm:ACtoBrakke}, predicts that the evolution effectively behaves like multiphase mean curvature flow. 
  \\
  Top: The function $\uu_\eps(\cdot,t)$ is almost constant in large parts of the domain, mostly taking values close to $\{\ee_1,\ldots,\ee_n\}$.  
  \\
  Bottom: The corresponding energy density $\omega^\eps_t$ (black corresponds to high energy) concentrates on the transition layers, close to a network of curves that moves to decrease its total length.}
  \label{fig:allen-cahn}
\end{figure}
We will show that any limit of solutions $\uu_\eps$ of the Allen--Cahn system~\eqref{eq:allen cahn} is a rectifiable Brakke flow according to the following definition, cf.~\cite{brakke}.
\begin{definition}[Rectifiable Brakke flow]\label{def:Brakke}
  We say that a weakly measurable family of mass measures $(\omega_t)_{t\in[0,\infty)}$ is a \emph{rectifiable Brakke flow with initial condition $\omega_0$} if
  \begin{enumerate}[(i)]
    \item For almost every $t>0$, $\supp \omega_t$ is $(d-1)$-rectifiable and the uniquely associated varifold $\mu_t= \omega_t \otimes (\delta_{T_x \omega_t})_{x\in \torus}$ has locally bounded first variation with mean curvature vector $\vec{H}_t$, i.e., $\delta \mu_t = -\vec{H}_t \omega_t$. Moreover, $\vec{H}_t$ points in normal direction, i.e., $(T_{x} \omega_t) \vec{H}_t(x) =0$ for $\omega_t$-a.e.\ $x\in \torus$.
    \item For any $0\leq t_1<t_2<\infty$ and any test function $\varphi\in C^1(\torus\times[0,\infty);[0,\infty))$ we have
    \begin{align}\label{eq:Brakke}
      \int_{\torus} \varphi \, \dd \omega_{(\cdot)}\Big|_{t_1}^{t_2} \leq \int_{t_1}^{t_2} \int_{\torus} \big(  -\varphi |\vec{H}_t|^2 + \vec{H}_t \cdot \nabla \varphi + \partial_t \varphi \big)\,\dd \omega_t\, \dd t.
    \end{align}
  \end{enumerate}
\end{definition}

Now we can rigorously formulate our main result which states that any weak limit of the Allen--Cahn system~\eqref{eq:allen cahn} is a Brakke flow.
\begin{theorem}\label{thm:ACtoBrakke}
  Let $\uu_\eps \colon \torus \times (0,\infty) \to \R^n$ be a solution to the Allen--Cahn system~\eqref{eq:allen cahn} with well-prepared initial conditions $\uu_{\eps,0}$ in the sense of Definition~\ref{def:well-prepared}.
  Then the sequence $\uu_\eps$ is pre-compact in $L^1$, the energy measures $\omega^\eps_t:= \big(\eps |\nabla \uu_\eps|^2 + F_\eps(\uu_\eps)\big) \,\dd x$ are pre-compact in the sense that any sequence $\eps\to0$ contains a subsequence $\eps'\to0$ such that $\omega^{\eps'}_t \stackrel{\ast}{\rightharpoonup} \omega_t $ in $C_c(\torus)^\ast$ for every $t\geq0$. 
  Any such limit point of $(\omega_t)_{t\geq0}$ is a rectifiable Brakke solution to multiphase mean curvature flow with initial condition
   \[
    \dd \omega_0 = \operatorname{weak}^\ast-\lim_{\eps\downarrow0} \Big( \frac\eps2 |\nabla \uu_{\eps,0}|^2 + F_\eps(\uu_{\eps,0})\Big) \,\dd x.
   \]
\end{theorem}

Our conditions on the initial data are the following.

\begin{definition}[Well-prepared initial data]\label{def:well-prepared}
  We say that initial data $\uu_{\eps,0}$ for the Allen--Cahn system are \emph{well-prepared} if the following conditions hold:
  \begin{enumerate}[(i)] 
    \item \label{item:pointwise_bounds} We have the pointwise bounds
 \begin{equation*}
    0\leq u_{0,\eps,i}\leq 1 \quad \text{in $\torus$ for all $i=1,\ldots,n$},
  \end{equation*}
  \item \label{item:Eeps0} The initial energy is uniformly bounded
  \begin{equation*}
    E_0:=\sup_{\eps \in (0,1]} E_\eps(\uu_{\eps,0}) <\infty.
  \end{equation*}
  \item We have the gradient bound
  \begin{align*}
    \sup_{\eps \in (0,1]} \sup_{\torus} \eps |\nabla \uu_{\eps,0}| <\infty.
  \end{align*}
  \item The initial discrepancy is bounded from above, i.e., there exists $C<\infty$ such that for any $i=1,\ldots,n$, and any $\eps\in(0,1]$,
  \begin{align*}
    \frac\eps2 |\nabla u_{\eps,0,i}|^2 - \frac1\eps W(u_{\eps,0,i}) \leq  C\eps^{-\frac{1+\alpha}{2}}\quad \text{in } \torus.
  \end{align*}
  \item The $(d-1)$-dimensional density of $\omega^\eps_0$ is bounded from above:
  \begin{align*}
    D_0:=\sup_{\eps\in(0,1]} \sup_{x\in\torus, r\in(0,1]} \frac{\omega^\eps_0(B_r(x))}{\omega_{d-1}r^{d-1}} <\infty.
  \end{align*}
  \end{enumerate}  
\end{definition}
In Appendix~\ref{sec:initial_data}, we construct initial data that satisfy all the conditions of Definition~\ref{def:well-prepared}, see Proposition~\ref{prop7} and Remark~\ref{rem:initialconditions}.
Note that the fact that $\vec{H}_t$ is a normal vector in Definition~\ref{def:Brakke} is crucial to recover Huisken's monotonicity formula for the limiting varifold from Brakke's inequality~\eqref{eq:Brakke}.

\medskip

Since Ilmanen's pioneering work~\cite{Ilmanen_allencahn}, a major obstacle for any type of vectorial Allen--Cahn system has been the vanishing of the discrepancy measure:
\begin{align}\label{eq:vanishing_of_discrepancy}
 \dd \xi^\eps := \Big( \frac\eps2|\nabla \uu_\eps|^2 - F_\eps(\uu_\eps)\Big) \,\dd x \,\dd t \stackrel{\ast}{\rightharpoonup} 0.
\end{align}
The crucial estimate in this is the \emph{upper} bound, more precisely, the vanishing of the positive part of the discrepancy measure:
\begin{align}\label{eq:vanishing_of_discrepancy_positive}
 \dd \xi^\eps_+ = \Big( \frac\eps2|\nabla \uu_\eps|^2 - F_\eps(\uu_\eps)\Big)_+ \,\dd x \,\dd t \stackrel{\ast}{\rightharpoonup} 0.
\end{align}
In the scalar case, Ilmanen proves the latter by a maximum principle argument. So, for suitable initial conditions, one even knows that $\dd \xi^\eps \leq 0$ in the scalar case. 
However, Ilmanen's proof does not seem to apply in the vectorial setting.
Instead, the key ingredient for Theorem~\ref{thm:ACtoBrakke} is the following result that shows exactly this vanishing of the positive part of the discrepancy by using integral estimates instead of pointwise estimates.
\begin{theorem}\label{thm:positive part of disrepancy measure}
  Let $\uu_\eps \colon \torus\times (0,\infty)\to \R^n$ be a solution to the Allen--Cahn system~\eqref{eq:allen cahn} with well-prepared initial data, then
  \begin{equation}\label{eq:integral discprepancy positive vanishes}
    \lim_{\eps\to0} \int_0^T\int_{\torus} \Big( \frac\eps2 |\nabla \uu_\eps|^2 - F_\eps(\uu_\eps)\Big)_+ \dd x \,\dd t =0 \quad \text{for all $T<\infty$.}
  \end{equation}
 In particular, the positive part of the discrepancy measure vanishes in the limit, i.e., 
  \begin{equation}\label{eq:positive_part_of_discrepancy_weak-star_to_zero}
    \dd \xi^\eps_+ = \Big( \frac\eps2 |\nabla \uu_\eps|^2 - F_\eps(\uu_\eps)\Big)_+ \dd x \,\dd t \stackrel{\ast}{\rightharpoonup} 0 \quad\text{as }\eps\to0.
  \end{equation}
\end{theorem}
The estimate \eqref{eq:integral discprepancy positive vanishes} from Theorem~\ref{thm:positive part of disrepancy measure} is crucial, as the discrepancy measure appears in the monotonicity formula (cf.\ Lemma~\ref{lemma:monotonicity formula} below). 
Such an estimate has not been available in the literature for any type of vectorial Allen--Cahn system.
While the resulting weak-$\ast$ convergence in~\eqref{eq:positive_part_of_discrepancy_weak-star_to_zero} allows testing against continuous functions, the kernel $\frac1{s-t}\rho_{(y,s)}(x,t)\chi_{\{0<t<s\}}$ against which the measure is integrated in the monotonicity formula has a singularity as $t\to s$. 
We prove that also these test functions are allowed in the convergence~\eqref{eq:positive_part_of_discrepancy_weak-star_to_zero}. This comes from more refined estimates on the positive part of the discrepancy measure, see Lemma~\ref{lemma9} and Corollary~\ref{cor:discrepancy decay} below. 
These refined estimates, together with the monotonicity formula then imply the vanishing of the negative part of the discrepancy measure, so that the full discrepancy measure vanishes in the limit.
\begin{theorem}\label{thm:discrepancy measure full}
 Under the conditions of Theorem~\ref{thm:positive part of disrepancy measure}, the discrepancy measure itself vanishes in the limit: 
  \begin{equation}
    \dd \xi^\eps \stackrel{\ast}{\rightharpoonup} 0 \quad\text{as }\eps\to0.
  \end{equation}
\end{theorem}

The rest of the paper is organized as follows. In Section~\ref{sec:monotonicity}, we derive a monotonicity formula that contains the discrepancy measure as an error. 
To handle this error, in Section~\ref{sec:discrepancy upper bound}, we derive upper bounds on the discrepancy measure and the vanishing of its positive part, and prove in particular Theorem~\ref{thm:positive part of disrepancy measure}.
In Section~\ref{sec:compactness}, we prove standard compactness of the energy measures and the convergence of the Allen--Cahn system to some limit partition.
Section~\ref{sec:vanishing_discrepancy} is devoted to the proof of the vanishing of the discrepancy measure (Theorem~\ref{thm:discrepancy measure full}). This is the key ingredient to prove that the limit is a Brakke flow (Theorem~\ref{thm:ACtoBrakke}), which in turn, is done in Section~\ref{sec:convergence}. 
In Appendix~\ref{sec:Gamma-convergence}, we check the $\Gamma$-convergence of the energy $E_\eps$ to the optimal partition energy $E$, and in Appendix~\ref{sec:initial_data} we construct well-prepared initial data for the Allen--Cahn system.

\begin{notation*}
We denote by $|s| = \max\{s,-s\}$, $s_+=\max\{s,0\}$, and $s_-=\max\{-s,0\}$ the absolute value, positive, and negative part of a real number $s$, respectively.
The constant $\omega_d$ denotes the volume of the $d$-dimensional unit ball.
We denote by $\torus$ the $d$-dimensional flat torus, i.e., the cube $[0,1)^d$ equipped with periodic boundary conditions.
For a matrix $A\in \R^{n \times d}$, we denote by $A^\ast \in \R^{d\times n}$ its transpose. 
For two vectors $a,b\in \R^d$ and two matrices $A,B \in \R^{d\times d}$ we denote by $a\cdot b = a^\ast b = \sum_{i=1}^d a_ib_i$ and $A\cdotdot B = \trace(A^\ast B) = \sum_{i}^n\sum_{j=1}^d A_{ij}B_{ij}$ the standard scalar products, the induced norms are denoted by $|a|= \sqrt{a\cdot a}$ and $|A| = \sqrt{ A\cdotdot A}$.
We write $a\otimes b= ab^\ast = (a_ib_j)_{ij}$ for the  tensor product of two vectors $a, b\in \R^d$. 
For a function $f\colon \torus \to \R$ we denote by $\nabla f = (\partial_j f)_{j=1,\ldots, d}\in\R^d$ its gradient.
For a vector field $\mathbf f = (f_i)_{i=1}^n \colon \torus \to \R^n$ we denote by $\nabla \mathbf f = (\partial_j f_i)_{i=1,\ldots,n, j=1,\ldots,d} \in \R^{n\times d}$ its Jacobi matrix.
The Hessian of a function is $\nabla^2 f = \nabla \nabla f \in \R^{d\times d}$.
The divergence of a vector field $\mathbf f$ is denoted by $\nabla \cdot \mathbf f = \sum_{i=1}^d\partial_i f_i$.
\end{notation*}

\section{Monotonicity formula and discrepancy measure}
\label{sec:monotonicity}
The gradient-flow structure manifested in~\eqref{eq:EDI_loosely} has important implications: The total energy is non-increasing in time and for all $T\in(0,\infty)$
\begin{equation}\label{eq:EDI}
  E_\eps(\uu_{\eps}(\cdot,T)) + \int_0^{T}  \int_{\torus} \frac1\eps \big|\eps\Delta\uu_\eps - \partial_\uu F_\eps(\uu_\eps)\big|^2 \,\dd x  \,\dd t = E_\eps(\uu_{\eps,0}).
\end{equation}
However, we want to localize this with a test function to obtain an $\eps$-version of Brakke's inequality.
\begin{lemma}\label{lemma:epsBrakke}
  Let $\uu \colon \torus\times(0,\infty)\to \R^n$ solve the Allen--Cahn system~\eqref{eq:allen cahn}, then the energy measure
  \[
    \dd \omega^\eps_t := \Big(\frac\eps2 |\nabla \uu(\cdot,t)|^2 + F_\eps(\uu(\cdot,t))\Big)\,\dd x
  \]
  satisfies
  \begin{align}
    \notag
    \frac{\dd}{\dd t} \int_{\torus} \varphi  \,\dd \omega_t^\eps
    =
    - &\int_{\torus} \varphi \, \frac1\eps \big|\eps\Delta \uu {-} \partial_{\uu}F_\eps(\uu)\big|^2 \,\dd x
    \\
    - &\int_{\torus} \sum_{i=1}^n (I{-}\nu_i\otimes \nu_i) \cdotdot \nabla^2 \varphi \, \eps |\nabla u_i|^2\, \dd x 
    + \int_{\torus}   \Delta \varphi \,  \dd \xi^\eps_t
    \label{eq:epsBrakkedt}
  \end{align}
  for any $\varphi\in C^2(\torus;[0,\infty))$ and almost every $t>0$, where $\nu_i = \frac{\nabla u_i}{|\nabla u_i|} \in \S^{d-1}$ denotes the normal to the level sets of $u_i$.
\end{lemma}

\begin{proof}[Proof of Lemma~\ref{lemma:epsBrakke}]
   We compute (identifying $\omega^\eps_t$ with its density)
  \begin{align}
    \notag
    \partial_t \omega^\eps_t
    &= \eps\nabla \uu \cdotdot \nabla \partial_t \uu +  \partial_{\uu}F_\eps(\uu)\cdot \partial_t \uu 
    \\
    \notag
    &= \eps \nabla \cdot \big( (\nabla \uu)^\ast \partial_t \uu\big) - \eps \Delta  \uu \cdot  \partial_t \uu + \partial_{\uu} F_\eps(\uu)\cdot \partial_t \uu 
    \\
    \label{eq:IBPenergydensity}
    &=  \nabla \cdot \Big( (\nabla \uu)^\ast \big(\eps\Delta \uu  - \partial_{\uu} F_\eps(\uu)\big)\Big)
     -  \frac1\eps \big|\eps\Delta \uu  - \partial_{\uu} F_\eps(\uu)\big|^2.
  \end{align}
  Using 
  \begin{align}
    \notag
    (\nabla \uu)^\ast (\eps \Delta \uu  -\partial_{\uu} F_\eps(\uu) )
  &= \nabla \cdot \big(\eps(\nabla \uu)^\ast \nabla \uu\big) - \nabla \big(\frac\eps2 |\nabla \uu|^2 \big) -\nabla (F_\eps \circ \uu)
  \\
  \label{eq:IBPLuckhausModica}
  &= \nabla \cdot \big( \eps (\nabla \uu)^\ast \nabla \uu\big) - \nabla \omega^\eps_t,
  \end{align}
  we obtain an identity for the backward diffusion operator on $\omega^\eps$:
  \begin{align*}
     \partial_t \omega_t^\eps  = - \Delta \omega^\eps_t + \nabla \cdot \big(\nabla \cdot\big( \eps (\nabla \uu)^\ast \nabla \uu\big)\big) - \frac1\eps \big|\eps\Delta \uu  - \partial_{\uu} F_\eps(\uu)\big|^2.
  \end{align*}
  Hence, when testing with $\varphi$, integrating by parts, and using $(\nabla \uu)^\ast \nabla \uu = \sum_{i=1}^n \nabla u_i \otimes \nabla u_i$, we obtain the claim.
\end{proof}

Following Huisken~\cite{Huisken_monotonicity}, we introduce the inverse heat kernel 
\begin{align}\label{eq:heat kernel}
  \rho_{(y,s)}(x,t) := \frac{1}{(4\pi (s-t))^{\frac{d-1}{2}}} \exp\Big(-\frac{|x-y|^2}{4(s-t)}\Big)
\end{align}
with the scaling of the dimension $d-1$ of our moving interfaces, which solves a variant of the backward diffusion equation
\begin{align}\label{eq:rho inverse heat equation}
  \partial_t \rho + (I-\nu\otimes \nu) \cdotdot \nabla^2 \rho + \frac{(\nu \cdot \nabla \rho)^2}{\rho} = 0
\end{align}
for any $\nu \in \mathbb{S}^{d-1}$.
Similar to Ilmanen~\cite{Ilmanen_allencahn}, we have the following almost-monotonicity formula for our Allen--Cahn system~\eqref{eq:allen cahn} in which the discrepancy measure appears as an error.

\begin{lemma}\label{lemma:monotonicity formula}
  Let $\uu \colon \torus\times(0,\infty)\to \R^n$ solve the Allen--Cahn system~\eqref{eq:allen cahn}, let $y\in \torus$ and $s>0$, and let $\rho_{(y,s)}$ denote the inverse heat kernel~\eqref{eq:heat kernel}. 
  Then, identifying $\omega^\eps_t$ and $\xi^\eps_t$ with their periodic extensions to $\R^d$, we have
\begin{align}
  \frac{\dd}{\dd t} \int_{\R^d} \rho_{(y,s)}(x,t)  \,\dd \omega^\eps_t(x)
  \leq
  \frac1{2(s-t)} \int_{\R^d} \rho_{(y,s)}(x,t)\,\dd \xi^\eps_t(x)
\end{align}
for any $t\in(0,s)$. Hence, decomposing $\xi^\eps_t = \xi^\eps_{t,+} - \xi^\eps_{t,-}$ into its positive and negative part, we get
\begin{align}
  \frac{\dd}{\dd t} \int_{\R^d} \rho_{(y,s)} (x,t) \,\dd \omega^\eps_t(x)
  &+\frac1{2(s-t)} \int_{\R^d} \rho_{(y,s)}(x,t) \,\dd \xi^\eps_{t,-} (x)
  \leq 
  \frac1{2(s-t)} \int_{\R^d} \rho_{(y,s)}(x,t)\,\dd \xi^\eps_{t,+} (x).
\end{align}
\end{lemma}

Here, again, just like in Ilmanen's work, the discrepancy measure needs to be controlled in order to extract information from this inequality. More precisely, we need an \emph{upper} bound on it.

\begin{proof}[Proof of Lemma~\ref{lemma:monotonicity formula}]
  Starting from the identity~\eqref{eq:IBPenergydensity} in the proof of Lemma~\ref{lemma:epsBrakke}, testing with $\rho=\rho_{(y,s)}$, setting $\ff_\eps:= -\big(\eps\Delta \uu  - \partial_{\uu} F_\eps(\uu)\big)$, and adding zero yields
  \begin{align*}
    \frac{\dd}{\dd t} \int_{\R^d} \rho \,\dd \omega^\eps_t 
    &= \int_{\R^d} \Big(- \rho  \frac1\eps |\ff_\eps |^2 
      +2 \nabla\rho \cdot (\nabla \uu)^\ast \ff_\eps
      + \nabla\rho \cdot \Big( (\nabla \uu)^\ast \big(\eps\Delta \uu  - \partial_{\uu} F_\eps(\uu)\big)\Big)
      + \partial_t \rho\,  \omega^\eps_t\Big).
  \end{align*}
  Completing the square and using~\eqref{eq:IBPLuckhausModica}, this is equal to
  \begin{align*}
    &\int_{\R^d} \Big(- \rho  \frac1\eps \Big|\ff_\eps - \eps\frac{\nabla \uu \nabla \rho}{\rho} \Big|^2 
      +\eps \frac{|\nabla \uu \nabla \rho|^2}{\rho}
      + \Delta \rho \, \omega^\eps_t - \eps \nabla^2 \rho \cdotdot (\nabla\uu)^\ast \nabla \uu
      + \partial_t \rho \, \omega^\eps_t\Big).
  \end{align*}
  Now we rewrite the remainder, denoting $\nu_i:=\frac{\nabla u_i}{|\nabla u_i|}$ and using $ \omega^\eps_t = \eps |\nabla \uu|^2 - \xi^\eps_t$, as
  \begin{align*}
      \eps \frac{|\nabla \uu \nabla \rho|^2}{\rho}
      &+ \Delta \rho \, \omega^\eps_t - \eps \nabla^2 \rho \cdotdot (\nabla\uu)^\ast \nabla \uu
      + \partial_t \rho \, \omega^\eps_t
      \\
      &= \sum_{i=1}^n \Big( \frac{|\nu_i \cdot \nabla \rho|^2}{\rho} 
         - \nabla^2 \rho \cdotdot \nu_i \otimes \nu_i + \partial_t \rho + \Delta \rho\Big) \eps |\nabla u_i|^2 - (\partial_t \rho +\Delta \rho)\xi^\eps_t.
  \end{align*}
  The term in the parentheses vanishes due to~\eqref{eq:rho inverse heat equation}, and we compute
  \[
    \partial_t \rho + \Delta \rho = -\frac{1}{2(s-t)} \rho,
    \]
  so that 
  \[
    \frac{\dd}{\dd t} \int_{\R^d} \rho \,\dd \omega^\eps_t 
    =
    -\int_{\R^d}\rho  \frac1\eps \Big|\ff_\eps - \eps\frac{\nabla \uu \nabla \rho}{\rho} \Big|^2 \,\dd x
    + \frac1{2(s-t)}\int_{\R^d} \rho \, \dd \xi^\eps_t,
  \]
  and in particular, the inequality from the lemma follows.
\end{proof}

\section{Proof of Theorem~\ref{thm:positive part of disrepancy measure}: Upper bound on the discrepancy}
\label{sec:discrepancy upper bound}
Before we prove the vanishing of the positive part of the discrepancy measure, let us first record a simple upper bound on it.

\begin{remark}
  We have
  \begin{align*}
    \frac\eps2 |\nabla \uu|^2 - F_\eps (\uu)
    &= \sum_{i=1}^n \Big( \frac\eps2 |\nabla u_i|^2 - \frac1\eps W(u_i)\Big) - \frac1{2\eps^\alpha} \Big( \sigma - \sum_{i=1}^n \phi (u_i)\Big)^2
    \\&
    \leq
    \sum_{i=1}^n \Big( \frac\eps2 |\nabla u_i|^2 - \frac1\eps W(u_i)\Big).
  \end{align*}
  Thus, in order to obtain~\eqref{eq:vanishing_of_discrepancy}, we only need to show
  \begin{align}
    \int_0^T \int_{\torus} \Big( \frac\eps2 |\nabla u_i|^2 - \frac1\eps W(u_i)\Big)_+ \,\dd x \, \dd t \longrightarrow 0 \quad \text{as $\eps\to0$} \quad \text{for all $i=1,\ldots,n$.} 
  \end{align}
\end{remark}

This remark allows us to resort to the following integral estimate due to Chen~\cite[Theorem 3.6]{Chen_global}, that gives an upper bound on the discrepancy measure in the scalar setting with a well-behaved right-hand side.

\begin{theorem}[Theorem 3.6 in~\cite{Chen_global}]\label{thm:chen}
  Let 
  \begin{align}
    \mathcal{K}_\eps := \Big\{ (u,v)\in H^2(\torus) \cap L^2(\torus) \,\Big|\, -\eps \Delta u +\frac1\eps W'(u) = v \text{ in } \torus\Big\}.
  \end{align}
  Then there exists a constant $\eta_0\in (0,1]$ and a 
  continuous, non-increasing function 
  \[
    M \colon (0,\eta_0] \to [0,\infty)
  \] 
  such that 
  for every $\eta \in (0,\eta_0]$, every $\eps \in (0,\frac1{M(\eta)}]$, 
  and every $(u_\eps,v_\eps) \in \mathcal{K}_\eps$
  \begin{align}
    \int_{\torus} \Big(\frac\eps2 |\nabla u_\eps|^2 - \frac1\eps W(u_\eps) \Big)_+\,\dd x
    \leq \eta \int_{\torus}\Big(\frac\eps2 |\nabla u_\eps|^2 +\frac1\eps W(u_\eps) \Big) \,\dd x
    +  M(\eta) \eps \int_{\torus} v_\eps^2\,\dd x.
  \end{align}
\end{theorem}
Chen's original statement contains two such functions $M_1$ and $M_2$. 
As the result is purely qualitative in these functions, we simply take $M=\max\{M_1,M_2\}$.
Moreover, his result is stated on a bounded domain with Neumann boundary conditions, but it applies verbatim in our case of periodic boundary conditions.

Theorem~\ref{thm:chen} becomes extremely useful if, next to a uniform energy bound 
\[
  \sup_{\eps\in (0,1]}\int_{\torus}\Big(\frac\eps2 |\nabla u_\eps|^2 +\frac1\eps W(u_\eps) \Big) \,\dd x <+\infty,
\] 
one has the vanishing of the limit
\begin{equation}
    \lim_{\eps\to0} \, \eps \int_{\torus} v_\eps^2\,\dd x  =0.
  \end{equation}
Indeed, then one can take first $\eps\to0$ and then $\eta \to0$ to obtain 
\[
  \lim_{\eps\to0}\int_{\torus} \Big(\frac\eps2 |\nabla u_\eps|^2 - \frac1\eps W(u_\eps) \Big)_+\,\dd x =0.
\]

This is precisely the proof strategy for Theorem~\ref{thm:positive part of disrepancy measure}, which in turn is the key to passing to the limit in the Allen--Cahn system.

\begin{proof}[Proof of Theorem~\ref{thm:positive part of disrepancy measure}]
  Let $i\in \{1,\ldots, n\}$ and set
  \[
    v_\eps := -\eps \partial_t u_{\eps,i}+ \frac{\sqrt{2W(u_{\eps,i})}}{\eps^\alpha} \Big(\sigma - \sum_{j=1}^n \phi(u_{\eps,j}) \Big).
  \]
  Then $(u_{\eps,i}(\cdot,t),v_\eps(\cdot,t)) \in \mathcal{K}_\eps$ for almost every $t\in(0,T)$.

  We claim that
  \begin{equation}\label{eq:estimate v_eps}
    \eps \int_0^T \int_{\torus} v_\eps^2 \,\dd x \, \dd t 
    \leq \big(2 \eps + 4(n-1) \sigma T \eps^{2-2\alpha}\big) E_\eps(\uu_{\eps,0}).
  \end{equation}
  Indeed, since $W'(0)=W'(1)=0$ and $\sqrt{2W(0)} =\sqrt{2W(1)} = 0$, by the maximum principle, we have
  \[
    0 \leq u_{\eps,j} \leq 1  \quad \text{for all } j=1,\ldots,n.
  \]
  Therefore, we have the uniform bound
  \[
    \Big|\sigma - \sum_{j=1}^n \phi(u_{\eps,j}) \Big| \leq (n-1)\sigma
  \]
  and hence by Young's inequality and the energy-dissipation inequality~\eqref{eq:EDI}
  \begin{align*}
    \int_0^T \int_{\torus} v_\eps^2 \,\dd x \, \dd t 
    &\leq 2\int_0^T\int_{\torus} \eps^2 |\partial_t \uu_\eps|^2\,\dd x\, \dd t
    + 2 (n-1)\sigma  T \sup_{t\in(0,T)} \int_{\torus} \frac2{\eps^{2\alpha}} W(u_{\eps,i})\,\dd x
    \\&
    \leq 
    2 E_\eps (\uu_{\eps,0}) + 4(n-1) \sigma T \frac{1}{\eps^{2\alpha-1}} E_\eps (\uu_{\eps,0}),
  \end{align*}
  which after multiplication by $\eps$ is precisely the claim~\eqref{eq:estimate v_eps}.

  Applying Theorem~\ref{thm:chen} for almost every time, integrating over $t\in(0,T)$, estimating the first right-hand side term by~\eqref{eq:EDI} and the second one by~\eqref{eq:estimate v_eps}, recalling the bound on the energy of the initial conditions, Definition~\ref{def:well-prepared}\eqref{item:Eeps0}, yields
  \[
    \int_0^T\int_{\torus} \Big( \frac\eps2 |\nabla \uu_\eps|^2 - F_\eps(\uu_\eps)\Big)_+ \dd x \,\dd t 
    \leq 
    \eta  T E_0 + M(\eta )\big(2 \eps + 4(n-1) \sigma T \eps^{2-2\alpha}\big) E_0
  \]
  for all $\eta\in(0,\eta_0]$ and $\eps\in (0,\min\{1,\frac1{M(\eta)}\}]$.
  Note that since $\alpha<1$, the second right-hand side term vanishes in the limit $\eps\to0$. 
  Hence, taking $\eps\downarrow0$ while keeping $\eta\in(0,\eta_0]$ fixed, we obtain
  \[
    \limsup_{\eps\to0}\int_0^T\int_{\torus} \Big( \frac\eps2 |\nabla \uu_\eps|^2 - F_\eps(\uu_\eps)\Big)_+ \dd x \,\dd t 
    \leq \eta T E_0.
  \]
  Since the left-hand side is independent of $\eta$, taking $\eta\to0$ yields the claim.
\end{proof}

\section{Compactness}
\label{sec:compactness}
In this section, we show the compactness of the functions, energy measures, and associated varifolds.
We first observe an almost-monotonicity of the energy measures.

\begin{lemma}[Semi-decreasing property]
  For all $\varphi \in C^2(\torus;[0,\infty))$ and all $\eps\in(0,1]$, the map
  \[
    (0,\infty) \ni t\mapsto \int_{\torus} \varphi \,\dd \omega^\eps_t - E_0\|\varphi\|_{C^2}t \quad \text{is monotonically non-increasing.}
  \]
\end{lemma}

\begin{proof}
  Note that for any $\varphi \in C^2(\torus;[0,\infty))$, it holds\footnote{For $\eps>0$ fixed, $f_\eps:=\frac{|\nabla \varphi|^2}{\varphi+\eps}$, attains a maximum on $\torus$, and there we have $0= \nabla f_\eps = \frac{2\nabla^2 \varphi \nabla \varphi}{\varphi +\eps} - \frac{|\nabla \varphi|^2 \nabla \varphi}{(\varphi+\eps)^2}$. 
  If $\varphi\equiv 0$, there is nothing to prove; otherwise, the maximum of $f_\eps$ is positive, 
  so $\nabla\varphi$ cannot vanish at the maximum. 
  Hence 
  $\frac{|\nabla\varphi|^2}{\varphi+\eps} 
  = 2 \big| \nabla^2\varphi \frac{\nabla \varphi}{|\nabla \varphi|}\big|$. Taking $\eps\to0$ yields the uniform upper bound.}
  \begin{align}\label{eq:Dphi/phi vs D2phi}
    \sup \frac{|\nabla \varphi|^2}{\varphi} \leq 2 \sup |\nabla^2 \varphi|.
  \end{align}
  We compute
  \begin{align*}
    \frac{\dd}{\dd t} \int_{\torus} \varphi \,\dd \omega^\eps_t 
    &= - \int_{\torus} \eps \nabla \varphi \cdot (\nabla \uu_\eps)^\ast \partial_t \uu_\eps -  \int_{\torus} \eps \varphi |\partial_t \uu_\eps|^2 \,\dd x
    \\
    &= 
    \int_{\torus} \eps \frac{ |\nabla \uu_\eps\nabla \varphi|^2}{4\varphi} \,\dd x
    - \int_{\torus} \eps \varphi \Big| \partial_t \uu_\eps +\frac{\nabla \uu_\eps \nabla \varphi}{2\varphi} \Big|^2 \, \dd x
    \\
    &\leq
     \int_{\torus}  \frac{ |\nabla \varphi|^2}{2\varphi} \,\dd \omega^\eps_t.
  \end{align*}
  So the claim follows from the pointwise inequality~\eqref{eq:Dphi/phi vs D2phi} and the energy-dissipation estimate~\eqref{eq:EDI}.
\end{proof}

This monotonicity together with a standard separability argument implies the following pre-compactness of the energy measures.
\begin{proposition}\label{prop:compactness of energy measures}
  There exists a family of Radon measures $\{\omega_t\}_{t\geq0}$ and a subsequence such that $\omega^{\eps_j}_t \stackrel{\ast}{\rightharpoonup} \omega_t$ for all $t\geq0$. 
\end{proposition}

\begin{proof}
  This follows as in~\cite[Section~5.4]{Ilmanen_allencahn}, see also the proof of~\cite[Proposition~5.2]{MizunoTonegawa}.
\end{proof}

Next, we show the weak-$\ast$ pre-compactness in $BV$ of the functions themselves.
\begin{proposition}\label{prop:compactness of u}
There exist $u _i \in BV _{loc} (\mathbb{T}^d \times [0,\infty)) \cap C^{\frac12} _{loc} ([0,\infty); L^1 (\mathbb{T}^d))$ with $\sum_{i=1}^n u_i = 1$ a.e., and a subsequence $\varepsilon _j \to 0$ such that $u_{\varepsilon_j,i} \to u_i$ in $L^1 _{loc} (\mathbb{T}^d \times [0,\infty))$ and a.e.\ pointwise. 
Moreover, $\sum_{i=1} ^n | \nabla u_i (\cdot,t)| (\varphi ) \leq \frac{1}{\sigma} \omega _t (\varphi)$ 
for all $t \in [0,\infty)$ and $\varphi \in C(\mathbb{T}^d;[0,\infty))$.
\end{proposition}
\begin{proof}
We define
\[
\Phi (s) = \frac1\sigma \phi (s) = \frac{\int_0 ^s \sqrt{2W (a)} \, \dd a}{\int_0 ^1 \sqrt{2W (a)} \, \dd a},
\quad 
\text{and}
\quad
w_{\varepsilon_j,i} = \Phi \circ u_{\varepsilon_j, i}.
\]
We compute
\begin{equation}
|\nabla w_{\varepsilon_j,i} |
=
\frac1\sigma |\nabla u_{\varepsilon_j,i}| \sqrt{2W(u_{\varepsilon_j,i})}
\leq 
\frac1\sigma
\left(
\frac{\varepsilon _j |\nabla u_{\varepsilon_j,i} |^2}{2} 
+\frac{W(u_{\varepsilon_j,i})}{\varepsilon_j} 
\right)
\label{eq5.9}
\end{equation}
and
\[
|\partial _t w_{\varepsilon_j,i}|
\leq 
\frac1\sigma
\left(
\frac{\varepsilon_j  |\partial_t u_{\varepsilon_j,i} |^2}{2} 
+\frac{W(u_{\varepsilon_j,i})}{\varepsilon_j} 
\right).
\]
Hence there exists $C>0$ such that for any $i=1,\dots,n$ we have
\[
\sup_{j\geq 1} \left\{ \max _{0\leq t \leq T} \int _{\mathbb{T}^d } |\nabla w_{\varepsilon_j,i} (x,t) | \, \dd x
+
\int _0 ^T \int _{\mathbb{T}^d} |\partial_t w_{\varepsilon_j,i} | \, \dd x \, \dd t \right\} \leq C.
\]
Thus $\{ w_{\varepsilon_j,i} \} _{j=1} ^\infty$
is bounded in $BV_{loc} (\mathbb{T}^d \times [0,T])$.
By $BV$ compactness, there exist a subsequence (denoted by the same index) and
$w_i \in BV_{loc} (\mathbb{T}^d \times [0,\infty))$
such that
\[
w_{\varepsilon_j,i} \to w_i 
\qquad
\text{strongly in } L^1 _{loc} (\mathbb{T}^d \times [0,\infty))
\] 
and a.e.\ pointwise. Since
\[
\sup _{j\geq 1, \, t\geq 0} \int _{\mathbb{T}^d} \frac{W (u_{\varepsilon_j,i} (x,t))}{\varepsilon _j} \, \dd x <\infty,
\]
we have $u_i(x,t) := \lim_{j\to \infty} u_{\varepsilon_j,i} (x,t) \in \{0,1\}$ for a.e.\ $(x,t)$, and hence
$w_i \in \{0,1\}$ for a.e.\ $(x,t)$. One can easily check that
$u_i =w_i$ a.e.\ on $\mathbb{T}^d \times [0,\infty)$ and thus $u_i \in BV_{loc} (\mathbb{T}^d \times [0,\infty))$.
Moreover, there exists $C >0$ such that for any $0\leq t_1 <t_2 <T$,
\begin{equation}
\begin{split} 
 \int _{\mathbb{T}^d} |w_{\varepsilon_j,i} (x,t_2) -w_{\varepsilon_j,i} (x,t_1)| \, \dd x 
&\leq  
\int _{\mathbb{T}^d} \int _{t_1} ^{t_2} |\partial _t w_{\varepsilon_j,i}| \, \dd t \, \dd x \\
&\leq  \frac1\sigma
\int _{\mathbb{T}^d} \int _{t_1} ^{t_2} 
\left(
\frac{\varepsilon_j  | \partial_t u_{\varepsilon_j,i} |^2}{2} \sqrt{t_2 -t} 
+\frac{W(u_{\varepsilon_j,i})}{\varepsilon_j \sqrt{t_2 -t} } 
\right)
  \dd t \, \dd x \\
& \leq  C \sqrt{t_2 -t_1} .
\end{split} 
\label{eq5.7}
\end{equation} 
By \eqref{eq5.7} and for a.e.\ $0\leq t_1 <t_2 <T$,
\[
\lim _{j \to \infty} \int _{\mathbb{T}^d } | w_{\varepsilon_j,i} (x,t_2) - w_{\varepsilon_j,i} (x,t_1)| \, \dd x
= \int _{\mathbb{T}^d} |w_i (x,t_2) -w_i (x,t_1)| \, \dd x,
\]
hence we can redefine $w_i$ on a set of measure zero so that $w_i \in C^{\frac12} _{loc} ([0,\infty) ;L^1 (\torus))$.
Moreover, \eqref{eq5.9} and $|\nabla u_i (\cdot,t)| =|\nabla w_i (\cdot,t)|$ imply 
$\sum_{i=1} ^n \int \varphi | \nabla u_i (\cdot,t)|  \leq \frac{1}{\sigma} \omega _t (\varphi )$ 
for all $t \in [0,\infty)$ and $\varphi \in C(\mathbb{T}^d;[0,\infty))$.

Finally, for any $t\geq0$ we have
\begin{align*}
  \int_{\torus} \frac{\sigma^2}{2\eps^\alpha} \Big(1 - \sum_{i=1}^n w_{\eps,i}\Big)^2 \,\dd x
  = \int_{\torus} \frac{1}{2\eps^\alpha} \Big(\sigma - \sum_{i=1}^n \phi(u_{\eps,i})\Big)^2 \,\dd x 
  \leq E_\eps(\uu_\eps) \leq E_0 <\infty.
\end{align*}
Hence, $\sum_{i=1}^n w_i(x,t) = 1$ for a.e.\ $x\in \torus$ and $t\geq0$. Since $w_i = u_i$ a.e., this implies that $\sum_{i=1}^n u_i = 1$ a.e.\ on $\torus \times [0,\infty)$.
\end{proof}

\section{Proof of Theorem~\ref{thm:discrepancy measure full}: Vanishing of the discrepancy}
\label{sec:vanishing_discrepancy}
In this section, we prove that the discrepancy measure vanishes in the limit $\varepsilon \to 0$.
We adapt the argument in \cite{TakasaoTonegawa_transport} to our system of Allen--Cahn equations.

Recall that due to our definition of well-prepared initial data, the following estimates are fulfilled:
\begin{equation}\label{gradient-est-initial}
0 \leq u_{i,0} \leq 1 \quad \text{on} \ \mathbb{T}^d, \quad
\sup_{\mathbb{T}^d } \varepsilon |\nabla \uu_0| \leq C
\qquad \text{for any} \ \varepsilon>0 \ \text{and} \ i =1,\dots, n,
\end{equation}
and
\begin{equation}\label{discrepancy0}
\frac{\varepsilon}2 |\nabla u_{i,0}|^2 -\frac1\eps W (u_{i,0}) \leq C \varepsilon ^{-\frac{1+\alpha}{2}} \qquad \text{on} \ \mathbb{T}^d .
\end{equation}
Then, using standard arguments, we can propagate these estimates to all positive times.
\begin{lemma}
There exists $C>0$ such that
\begin{equation}\label{gradient-est}
0 \leq u_{i} \leq 1  \quad \text{in} \ \mathbb{T}^d \times [0,\infty), \quad
\sup_{\mathbb{T}^d \times [0,\infty) } \varepsilon |\nabla \uu| \leq C
\qquad \text{for any} \ \varepsilon>0 \ \text{and} \ i =1,\dots, n.
\end{equation}
\end{lemma}
\begin{proof}
Since $W'(0)=W'(1)=\sqrt{2W(0)}=\sqrt{2W(1)}=0$, $\underline{u} (x,t):=0$ and $\overline{u} (x,t):=1$ are sub- and the super-solutions to \eqref{eq:allen cahn}, respectively.
Hence, $0 \leq u_{i} \leq 1$ by the maximum principle and \eqref{gradient-est-initial}.

Next, we show the gradient estimate of $\uu$. By a standard calculation, we obtain
\[
\frac{1}{2} \partial _t |\nabla \uu|^2 
=
\frac12 \Delta |\nabla \uu|^2 - |\nabla^2 \uu|^2 
- \frac1\eps \sum_{i,j =1}^n \sum_{k=1}^d \partial _{x_k} u_i \partial _{x_k} u_{j} \partial_{u_i, u_j} F_\varepsilon (\uu), \qquad (x,t) \in \mathbb{T}^d \times (0,\infty).
\]
By the definition of $F_\varepsilon$, there exists $K>0$ such that $|\partial_{\uu}^2 F_\varepsilon (\uu)| \leq K/\varepsilon$ for any $\uu \in \mathbb{R}^n$.
Set $w(x,t):= e^{- \frac{Kt}{\varepsilon^2}} |\nabla \uu(x,t)|^2$. Then we have
\[
\partial _t w \leq \Delta w, \qquad (x,t) \in \mathbb{T}^d \times (0,\infty)
\]
and thus the maximum principle implies that 
\[
|\nabla \uu (x,t)| \leq e^{\frac{K}{2}} \max_{x \in \mathbb{T}^d} |\nabla \uu (x,0)|
\leq \frac{C}{\eps} e^{\frac{K}{2}} ,
\qquad (x,t) \in \mathbb{T}^d \times [0,\varepsilon^2],
\]
where we used~\eqref{gradient-est}. 
For $t \geq \varepsilon^2$, the gradient estimate can be obtained by a standard interior gradient estimates via parabolic $L^p$ regularity and Sobolev embeddings using a rescaling argument. See, for example, \cite[Lemma~4.1]{TakasaoTonegawa_transport} (note, however, that since this lemma involves estimates down to the initial time, it needs to be adapted as an interior estimate).
\end{proof}

In the next lemma, we prove uniform upper bounds on the discrepancy measure, based on an idea of Hutchinson and Tonegawa~\cite{HutchinsonTonegawa} in the case stationary points of the scalar Allen--Cahn equation.
\begin{lemma}
There exist $\eps _1 >0$ and $C>0$ such that for any $\varepsilon \in (0,\eps_1)$ and for any $i \in \{1,\dots,n\}$ we have
\begin{equation}\label{discrepancy1}
\frac{\varepsilon}2 |\nabla u_i|^2 -\frac1\eps W (u_i) \leq C \varepsilon ^{-\frac{1+\alpha}{2}} \qquad \text{on} \ \mathbb{T}^d \times [0,\infty).
\end{equation}
\end{lemma}

\begin{proof}
By the rescaling $\tilde u_i (x,t)= u_i (\varepsilon x, \varepsilon ^2 t)$, we have
\begin{equation}\label{rescaledeq}
\partial_t \tilde u_i 
  = \Delta \tilde u_i - W'( \tilde u_i) + \eps^{1-\alpha} \sqrt{2W(\tilde u_i)} g(\tilde \uu),
\end{equation}
where $g(\tilde \uu):=\Big(\sigma - \sum_{j=1}^n \phi(\tilde u_j)\Big)$.
Hereafter, the tilde on the function $\tilde u_i$ is omitted.

Set 
\[
\zeta_i := \frac12 |\nabla u_i|^2 -W(u_i) -G(u_i),
\]
where we will choose $G$ later. 
By a direct calculation and~\eqref{rescaledeq}, we have
\begin{align*}
\partial _t \zeta_i -\Delta \zeta_i 
&=  
\varepsilon ^{1-\alpha} \{ \nabla u_i \cdot \nabla (g(\uu) \sqrt{2W (u_i)}) - (W' (u_i) + G'(u_i)) g(\uu) \sqrt{2W (u_i)}\} \\
& ~~~+ (W'(u_i) + G'(u_i)) W'(u_i) - | \nabla^2 u_i |^2 + G'' (u_i) |\nabla u_i| ^2.
\end{align*}
By \eqref{gradient-est}, for the rescaled function we have $|\uu|\leq C$ and $|\nabla \uu|\leq C$. Therefore
\begin{equation*}
\begin{split}
\big|\nabla u_i \cdot \nabla (g(\uu) \sqrt{2W (u_i)}) - (W' (u_i) + G'(u_i)) g(\uu) \sqrt{2W (u_i)}\big| \leq C 
\end{split}
\end{equation*}
for some constant~$C>0$, where we used~$\frac{|W' (s)|}{\sqrt{2W(s)}} \leq C|s-\frac12|$.
In addition, by the Cauchy--Schwarz inequality, we have
\begin{equation*}
\begin{split}
& |\nabla \zeta_i |^2 + 2(W'(u_i) + G'(u_i)) \nabla \zeta_i \cdot \nabla u_i 
+ (W'(u_i) + G'(u_i))^2 |\nabla u_i|^2 \\
&~~~~=  | \nabla \zeta_i + (W' (u_i) + G' (u_i)) \nabla u_i |^2 \\
&~~~~=  | \nabla^2 u_i \nabla u_i|^2 
\leq |\nabla^2 u_i|^2|\nabla u_i|^2 .
 \end{split}
\end{equation*}
Hence, on~$(x,t)$ with~$|\nabla u_i (x,t)|\not=0$, we have
\begin{equation}\label{discrepancy2}
\begin{split}
&\partial _t \zeta_i -\Delta \zeta_i \\
\leq & \, 
\varepsilon ^{1-\alpha} \{ \nabla u_i \cdot \nabla (g(\uu) \sqrt{2W (u_i)}) - (W' (u_i) + G'(u_i)) g(\uu) \sqrt{2W (u_i)}\} \\
& + G'' (u_i) |\nabla u_i| ^2 
-\frac{2(W'(u_i) +G'(u_i)) \nabla \zeta \cdot \nabla u_i }{|\nabla u_i|^2}
-W'(u_i) G'(u_i) - (G'(u_i))^2 \\
\leq & \, 
\varepsilon ^{1-\alpha} C + G'' (u_i) |\nabla u_i|^2 
+\frac{2 |W'(u_i) +G'(u_i)| |\nabla \zeta|  }{|\nabla u_i|}
-W'(u_i) G'(u_i) - (G'(u_i))^2.
\end{split}
\end{equation}
Set 
\[
G(u_i) := \varepsilon ^{\frac{1-\alpha}{2}} \left\{ 1 -\frac18 \left(u_i -\frac12\right)^2\right\}.
\]
To get a contradiction, we assume there exists~$i\in \{1,\dots,n\}$ such that
\[
\sup_{ \varepsilon^{-1} \mathbb{T}^d \times [0,\infty)} \zeta_i \geq 2L \varepsilon ^{\frac{1-\alpha}{2}},
\]
where~$L>0$ will be chosen later.
We may assume that there exists~$\tilde T>0$ such that
\[
\max_{ \varepsilon^{-1} \mathbb{T}^d \times [0,\tilde T]} \zeta_i \geq L \varepsilon ^{\frac{1-\alpha}{2}}.
\]
Consider a maximum point~$(x_0,t_0)$ of~$\zeta_i $ on~$\varepsilon^{-1} \mathbb{T}^d \times [0,\tilde T]$. 
By \eqref{discrepancy0}, 
\[
\max_{ \varepsilon^{-1} \mathbb{T}^d } \zeta_i (\cdot,0) \leq C \varepsilon ^{\frac{1-\alpha}{2}}.
\]
Thus, by taking a sufficiently large constant~$L$, we may assume~$t_0 >0$. 
Hence, at~$(x_0,t_0)$, we have
\[
  \partial_t  \zeta_i \geq 0, \quad
  \nabla \zeta_i =0, \quad
  \text{and} \quad
  \Delta \zeta_i \leq 0.
\]
By this and \eqref{discrepancy2}, at $(x_0,t_0)$, we have
\begin{equation}\label{discrepancy3}
  \begin{split}
    0 \leq & \, 
    \varepsilon ^{1-\alpha} C + G'' (u_i) |\nabla u_i|^2 
    -W'(u_i) G'(u_i) - (G'(u_i))^2\\
    = & \, 
    \varepsilon ^{1-\alpha} C - \frac{1}{4} \varepsilon^{\frac{1-\alpha}{2}} |\nabla u_i|^2 
    -W'(u_i) G'(u_i) - (G'(u_i))^2 \\
    \leq & \, 
    \varepsilon ^{1-\alpha} C - \frac{1}{4} L \varepsilon^{1-\alpha} 
    -W'(u_i) G'(u_i) - (G'(u_i))^2,
  \end{split}
\end{equation}
where we used~$G''=-\frac14 \varepsilon ^{\frac{1-\alpha}{2}}$ and~$|\nabla u_i|^2 \geq \zeta_i \geq L\varepsilon ^{\frac{1-\alpha}{2}}$ at $(x_0,t_0)$. 
Since~$0\leq u_i\leq 1$, we have~$-W' G' \leq 0$. 
Then we get a contradiction for sufficient large~$L>0$ and sufficient small~$\varepsilon$. 
Hence, we obtain 
\[
  \left( \frac{|\nabla u_i|^2}{2} -W(u_i) \right) =\zeta_i +G(u_i) \leq (2L+1) \varepsilon ^{\frac{1-\alpha}{2}}. 
\]
Transforming this inequality back to the original scale yields the claim.
\end{proof}

Let $D_0$ be the constant in Definition~\ref{def:well-prepared} and define
\[
D_{1} := \frac{4(4\pi)^{\frac{d-1}{2}}}{e^{-\frac{1}{4}}\omega_{d-1}} D_{0} + 1.
\]
By Definition~\ref{def:well-prepared}, we have
$
D_0, D_1<\infty.
$ 
Let~$0<T<\infty$. We define
\[
  D_\varepsilon (t):= 
      \sup_{r \in (0,1), x \in \R^d} \frac{\omega_t ^\varepsilon (B_r (x))}{\omega_{d-1} r^{d-1}} ,
  \qquad t \in [0,T]
\]
and
\begin{equation}
T_\varepsilon :=
\sup \{ \tau \in [0,T] \mid \sup_{t \in [0,\tau)} D_\varepsilon (t) \leq D_1\}.
\end{equation}
Then, by $D_0<D_1$, for any $\varepsilon \in (0,1)$ it holds that $T_\varepsilon >0$.
We will prove that there exists $\eps_0 \in (0,1)$ such that
\begin{equation}
T_\varepsilon = T \qquad \text{for any} \ \varepsilon \in (0,\eps_0).
\end{equation}
More precisely, we will show in Theorem~\ref{thm:density} the upper bound on the density
\begin{equation}\label{upperbound of density}
\sup_{t \in [0,T)} \sup_{\varepsilon \in (0,\eps_0)} D_\varepsilon (t) \leq D_1.
\end{equation}

We first recall the following general lemma of Ilmanen~\cite[Lemma~3.4(i)--(iii)]{Ilmanen_allencahn}.
\begin{lemma}\label{lemma:ilmanen}
  Let $\omega$ be a Radon measure on $\torus$ satisfying 
  \[
    \frac{\omega(B_r(x))}{\omega_{d-1} r^{d-1}} \leq D_0 \quad \text{for all }x\in \torus\text{ and } r>0.
  \]
  Then the following hold:
  \begin{enumerate}[(i)]
    \item \label{item:ilmanen1} $\int_{\R^d} \rho_{(y,s)}(x,t) \, \dd\omega(x) \leq D_0$  for all $y\in \torus$ and $s>t\geq 0$.
    \item \label{item:ilmanen2}
    $
    \int_{\R^d\setminus B_R(y)} \rho_{(y,s)}(x,t) \, \dd\omega(x) 
    \leq 
    2^{d-1}  e^{-\frac{3R^2}{16(s-t)}}D_0
    $ for all $y\in \torus$ and $s>t\geq 0$.
    \item \label{item:ilmanen3} For any $\delta>0$ there exists $\gamma=\gamma(\delta)>0$ such that if $s>t\geq 0$, $y,y'\in \torus$, and $R>0$ satisfy $|y-y'|\leq \gamma (s-t)^\frac12$, then
    \[
      \int_{\R^d} \rho_{(y',s)}(x,t) \, \dd\omega(x) \leq 
      (1+\delta)\int_{\R^d} \rho_{(y,s)}(x,t) \, \dd\omega(x) +\delta D_0.      
    \]
  \end{enumerate}
\end{lemma}

Then, in our situation, we have the following basic estimate.
\begin{lemma}\label{lem6}
Assume $s,R,r>0$ satisfy $0\leq s-(\frac{R}{r})^2 \leq T_\varepsilon$. 
Define $\tilde s =  s-(\frac{R}{r})^2 $. Then, for any $y\in \mathbb{T}^d$, it holds that
\begin{equation}
\begin{split}
\int _{\R^d} \rho_{(y,s)} (x,\tilde s) \, \dd\omega_{\tilde s} ^\varepsilon (x)
\leq \left( \frac{r}{\sqrt{4\pi} R} \right)^{d-1} 
\omega_{\tilde s} ^\varepsilon (B_R(y)) 
+ 2^{d-1} D_1 e^{-\frac{3 r^2}{16}}.
\end{split}
\end{equation}
\end{lemma}

\begin{proof}
By the definition of $\rho$, one can check that
\begin{equation}
\begin{split}
\int _{B_R (y)} \rho_{(y,s)} (x,\tilde s) \, \dd\omega_{\tilde s} ^\varepsilon (x)
= & \, \left( \frac{r}{\sqrt{4\pi} R} \right)^{d-1} 
\int _{B_R (y)} e^{-\frac{r^2 |x-y|^2}{4R^2}} \, \dd\omega_t ^\varepsilon (x) \\
\leq & \, \left( \frac{r}{\sqrt{4\pi} R} \right)^{d-1} \omega_{\tilde s} ^\varepsilon (B_R (y)).
\end{split}
\end{equation}
By Lemma~\ref{lemma:ilmanen}\eqref{item:ilmanen2}, we have
\begin{equation}
\int_{\R^d \setminus B_R (y)} \rho _{(y,s)} (x,\tilde s) \, \dd\omega_{\tilde s} ^\varepsilon 
\leq 2^{d-1} D_\varepsilon (\tilde s) e^{-\frac{3R^2}{8 \cdot 2(R/r)^2}}
\leq 2^{d-1} D_1 e^{-\frac{3r^2}{16}},
\end{equation}
which concludes the proof.
\end{proof}

Hereafter, we will need our assumption $0<\alpha<1$ from Definition~\ref{def:well-prepared} and denote
\[
\beta= \frac{1+\alpha}{2} \quad \text{and} \quad \beta' =\frac{1+\beta}{2}.
\]
Note that $0<\beta<\beta'<1$.
The next lemma provides a lower bound on the ($d-1$)-dimensional density of $\omega_t^\varepsilon$.
\begin{lemma}\label{lemma43}
There exist $c_1 \in (0,1)$, $c_2 \geq 1$, and $\varepsilon_{2}$ with $0 <\varepsilon_{2} \le \varepsilon_{1}$ such that
if $\varepsilon \in (0, \varepsilon_{2})$ and for some $i\in \{ 1,\dots,n \}$ it holds that $\frac14 < u_i (y, s) < \frac34$ with $s \in (0,T_{\varepsilon})$, then
for any $t \in [0, T_{\varepsilon}]$ with $\max \{0,s-2\varepsilon^{2\beta^{\prime}}\} \le t \le s$, we have
\begin{equation*}
c_1 \le \frac{1}{R^{d-1}} \omega _{t}^{\varepsilon} (B_{R}(y)),
\end{equation*}
where $R=c_2 (s+\varepsilon^{2}-t)^{1/2}$.
\end{lemma}
\begin{proof}
Assume that $\frac14 < u_i (y,s)<\frac34$ for some $i \in \{1,\dots,n\}$.
Then we compute
\[
\int_{\mathbb{R}^d} \frac{e^{-\frac{|\tilde{x}|^2 }{4} }}{(4\pi )^{d-1}} W(\tilde u_i) \, \dd\tilde{x}
\leq
\int_{\mathbb{R}^d} \rho_{(y,s+\varepsilon^2)} (x,s) \, \dd\omega_{s}^{\varepsilon} (x),
\]
where $\tilde u_i (\tilde{x}, s) = u_i (y+\varepsilon \tilde{x}, s)$.
From the gradient estimate \eqref{gradient-est}, we have $\sup_{\tilde x, s}|\nabla_{\tilde x} \tilde u_i |\leq C$ for some constant $C>0$.
By this and $\frac14 < u_i(0,s) <\frac34$, there exists $a \in (0,1)$ such that 
\begin{equation}\label{422}
5a \le \int_{\R^d} \rho_{(y,s+\varepsilon^2)} (x,s) \, \dd\omega_{s}^{\varepsilon} (x),
\end{equation}
since there exists $\tilde R>0$ such that $\frac18 < u_i(\tilde x,s) <\frac78$ for any $\tilde x \in B_{\tilde R} (0)$ and thus there exists $c>0$ such that $W(\tilde u_i )>c$ on $B_{\tilde R}(0)$.
From the monotonicity formula and the pointwise estimate \eqref{discrepancy1}, for $\rho=\rho _{(y,s+\varepsilon^2)} (x,s)$, we have
\begin{equation*}
\begin{split}
\left. \int_{\R^d } \rho \,  \dd\omega_{\tau}^{\varepsilon} \right|_{\tau = t}^{s}  
\le& \int_{t}^{s} \int_{\mathbb{R}^d} \frac{\rho}{2(s+\varepsilon^{2}-\tau)} (\xi_{\tau})_+ \, \dd x\, \dd\tau\\
\le& C \int_{t}^{s} \frac{\varepsilon^{-\beta}}{\sqrt{s+\varepsilon^{2}-\tau}} \, \dd\tau 
\leq C \varepsilon^{-\beta} \sqrt{s+\varepsilon ^2 -t}
\\
\le&C \varepsilon^{\beta^{\prime}-\beta}.
\end{split}
\end{equation*}
Here we used $\int_{\mathbb{R}^d} \rho \,  \dd x \le \sqrt{4\pi (s+\varepsilon^{2}-\tau)}$ and $2\varepsilon ^{2\beta'} \geq s-t$ by $s-2\varepsilon ^{2\beta'} \leq t$.
Therefore
\begin{equation}\label{424}
\begin{split}
\int_{\mathbb{R}^d} \rho \,  \dd\omega_{s}^{\varepsilon} \le & 
\int_{\mathbb{R}^d} \rho \, \dd\omega_{t}^{\varepsilon} + C \varepsilon^{\beta^{\prime}-\beta}.
\end{split}
\end{equation}
By \eqref{422} and \eqref{424}, for sufficiently small $\varepsilon$ we have
\begin{equation}\label{426}
2a \le \int_{\mathbb{R}^d} \rho \,  \dd \omega_{t} ^{\varepsilon}.
\end{equation}
Next we apply Lemma \ref{lem6} with $r=\sqrt{\frac{16}{3}\log (2^{d-1} D_{1} a^{-1})}$. 
We note that $2^{d-1} D_{1}e^{-\frac{3 r^{2}}{16}}=a$. 
Then we have
\begin{equation}\label{427}
2a \leq 
\int_{\mathbb{R}^d} \rho \, \dd\omega _{t}^{\varepsilon} \le \left( \frac{r}{\sqrt{4\pi}R} \right)^{d-1} \omega_{t}^{\varepsilon} (B_{R}(y) )  +2^{d-1} D_{1} e^{-\frac{3 r^{2}}{16}}
=\left( \frac{r}{\sqrt{4\pi}R} \right)^{d-1} \omega_{t}^{\varepsilon} (B_{R}(y) )  +a. 
\end{equation}
Note that $(s+\varepsilon ^2) - t =(\frac{R}{r})^2$, that is, 
$R=r(s+\varepsilon^{2}-t)^{1/2}$.
From \eqref{426} and \eqref{427}, for sufficiently small $\varepsilon$ we obtain
\[
a \le \left( \frac{r}{\sqrt{4\pi} R} \right)^{d-1} \omega_{t}^{\varepsilon} (B_{R}(y)).
\]
This completes the proof.
\end{proof}

Set
\[
\tilde \xi (x,t):=\sum_{i=1} ^n \left( \frac{\varepsilon |\nabla u_i (x,t)|^2}{2} -\frac{W(u_i(x,t))}{\varepsilon}\right) .
\] 
Next we give a sharp estimate for the positive part of the discrepancy measure.
\begin{lemma}\label{lem4.4}
There exist $\varepsilon_{3} \in (0, \varepsilon_{2})$ and $C>0$ such that for any $y\in \mathbb{T}^d$, $r \in (\varepsilon^{\beta^{\prime}} , \infty)$, and $t \in [2\varepsilon^{2\beta^{\prime}}, T_{\varepsilon}]$, we have
\begin{equation}\label{431}
\int_{B_{r}(y) } ( \tilde \xi (x,t) )_{+} \, \dd x \le C \varepsilon^{\beta^{\prime}-\beta} r^{d-1},
\end{equation}
provided $0<\varepsilon<\varepsilon_{3}$.
\end{lemma}
\begin{proof}
Let $y \in \mathbb{T}^d$, $r \in  (\varepsilon^{\beta^{\prime}} , \infty)$ and fix $t_{\ast} \in [2\varepsilon^{\beta^{\prime}} , T_{\varepsilon}]$.
We denote
\[
\tilde{A} :=\Big\{ x \in B_{2r}(y) \,\Big|\, \frac14 < u_i (x, \tilde{t}) <\frac34 \ \text{for some} \ \tilde{t} \in [t_{\ast}-\varepsilon^{2\beta^{\prime}},  t_{\ast}] \ \text{and for some} \ i\in\{1,\ldots,n\} \Big\}
\]
and
\[
A:=\{ x \in B_{2r+2c_{2}\varepsilon^{\beta^{\prime}}} (y) \mid \text{dist}(x, \tilde{A}) < 2c_{2}\varepsilon^{\beta^{\prime}} \} .
\]
By Vitali's covering theorem applied to $\mathcal{F}=\{ \bar{B}_{2c_{2}\varepsilon^{\beta^{\prime}}} (x) | x \in \tilde{A} \} $, there exists a set of pairwise disjoint balls 
$\{ \bar{B}_{2c_{2}\varepsilon^{\beta^{\prime}}} (x_{j}) \}_{j=1}^{N}$ such that 
\begin{equation*}
x_{j} \in \tilde{A} \; \text{for each} \; j=1, \ldots , N \; \text{and} \; A \subset \bigcup_{j=1}^{N} \bar{B}_{10c_{2}\varepsilon^{\beta^{\prime}}} (x_{j}).
\end{equation*}
By the definition of $\tilde{A}$, for any $j=1,\dots,N$ there exists $\tilde{t_{j}}$ and $ i(j) \in \{1,\dots,n\}$ such that 
\begin{equation}\label{432}
t_{\ast} - \varepsilon^{2\beta^{\prime}} \le \tilde{t_{j}} \le t_{\ast} 
\quad
\text{and}
\quad
 \frac14 < u_{i(j)} (x_{j}, \tilde{t_{j}})< \frac34.
\end{equation}
We define $\hat{t} :=t_{\ast} -2\varepsilon^{2\beta^{\prime}}.$
Note that $\hat{t} \geq 0$ since $t_{\ast} \geq 2\varepsilon^{2\beta^{\prime}}.$ 
By \eqref{432}, 
\[
\varepsilon^{2\beta^{\prime}}  \le \tilde{t_{i}}-\hat{t}  \le 2\varepsilon^{2\beta^{\prime}} .
\]
Applying Lemma \ref{lemma43} for $s=\tilde{t_{j}}, y=x_{j}, t=\hat{t}$ and $R_{j}:=c_{2}(\tilde{t_{j}} + \varepsilon^{2} - \hat{t})^{\frac{1}{2}}$, we obtain
\[
c_{1}R_{j}^{d-1} \le \omega_{\hat{t}}^{\varepsilon} (B_{R_{j}} (x_{j})), \; j=1, \ldots , N.
\]
For $c_{3}:=c_{1}c_{2}^{d-1},$ we have
\[
c_{3}\varepsilon^{\beta^{\prime} (d-1)} \le \omega_{\hat{t}}^{\varepsilon} (B_{2c_{2}\varepsilon^{\beta^{\prime}}} (x_{j})),
\]
because $c_{2}\varepsilon^{\beta^{\prime}} \le c_{2}(\varepsilon^{2\beta^{\prime}} +\varepsilon^{2} )^{\frac{1}{2}} \le R_{i} \le 2c_{2}\varepsilon^{\beta^{\prime}}$ by \eqref{432}.
Since the balls $\{ B_{2c_{2}\varepsilon^{\beta^{\prime}}}(x_{j}) \}_{j=1}^{N} $ are pairwise disjoint and $B_{2c_{2}\varepsilon^{\beta^{\prime}}}(x_{j}) \subset B_{2r+2c_{2}\varepsilon^{\beta^{\prime}}}(y),$ we have
\[
Nc_{3}\varepsilon^{\beta^{\prime} (d-1)} \le \omega_{\hat{t}}^{\varepsilon} (B_{2r+2c_{2}\varepsilon^{\beta^{\prime}}} (y)).
\]
Hence by $A \subset \bigcup_{j=1}^{N} \bar{B}_{10c_{2}\varepsilon^{\beta^{\prime}}} (x_{j}),$ we have
\[
\mathcal{L}^{d} (A) \le N \omega_{d} (10c_{2}\varepsilon^{\beta^{\prime}} )^{d} \le \frac{\omega_{d} (10c_{2})^{d} \varepsilon^{\beta^{\prime}}}{c_{3}} \omega_{\hat{t}}^{\varepsilon} (B_{2r+2c_{2}\varepsilon^{\beta^{\prime}}} (y)).
\]
By $r \geq \varepsilon^{\beta^{\prime}}$ and the definition of $D_1$, 
\begin{equation*}
\mathcal{L}^{d} (A) \le \frac{\omega_{d} (10c_{2})^{d} \varepsilon^{\beta^{\prime}}}{c_{3}} D_{1}\omega_{d-1}(2r+2c_{2}\varepsilon^{\beta^{\prime}})^{d-1} \le c_{4}\varepsilon^{\beta^{\prime}}r^{d-1}, 
\end{equation*}
where $c_{4}:= \omega_{d} \omega_{d-1} (10c_{2})^{d}(2+2c_{2})^{d-1}D_{1}c_{3}^{-1}$.
By this and $\tilde \xi\leq \varepsilon ^{-\beta}$, 
\begin{equation}\label{434}
\int_{A \cap B_{r}(y)} ( \tilde \xi(x,t_\ast))_{+} \, \dd x \le \mathcal{L}^{d} (A) C\varepsilon^{-\beta} \le C c_{4}\varepsilon^{\beta^{\prime}-\beta}r^{d-1}.
\end{equation}
Next define $\varphi \in \text{Lip}(B_{2r}(y)) $ such that
\[
\varphi (x) :=
\begin{cases}
1 & \text{if }  x \in B_{r}(y) \backslash A, \\
0 & \text{if }  \text{dist}(x, B_{r}(y) \backslash A) \geq \varepsilon^{\beta^{\prime}}, 
\end{cases}
\] 
and
\[
|\nabla \varphi| \le 2\varepsilon^{-\beta^{\prime}}
\quad
\text{and}
\quad
0\le \varphi \le 1.
\]
We note that $\mathrm{supp} \,  \varphi \cap \tilde{A} = \emptyset$ since $r \geq \varepsilon^{\beta^{\prime}} $ and $2c_{2}\varepsilon^{\beta^{\prime}} >\varepsilon^{\beta^{\prime}}$.
Thus,
\begin{equation*}
u_i (x,s) \in [0,\frac14] \cup [\frac34,1] \quad 
\text{for any} \ x \in \mathrm{supp} \,  \varphi , \, s \in [t_{\ast}-\varepsilon^{2\beta^{\prime}}, t_{\ast}], \ \text{and} \ i =1,\dots,n.
\end{equation*}
Hence, there exists $C>0$ such that
\begin{equation}
W'' (u_i (x,s)) \geq C \qquad \text{for any} \ x \in \mathrm{supp} \,  \varphi , \, s \in [t_{\ast}-\varepsilon^{2\beta^{\prime}}, t_{\ast}], \ \text{and} \ i =1,\dots,n.
\end{equation}
For each $k=1,\dots, d$ differentiate the equation \eqref{eq:allen cahn} with respect to $x_{k}$, multiply $\varphi^{2}\frac{\partial u_i }{\partial x_{k}} $, sum over $k$ and integrate to obtain
\begin{equation}\label{436}
\begin{split}
& \, \frac{1}{2} \frac{\dd}{\dd t} \int_{\mathbb{R}^d} \varphi^{2} |\nabla \uu|^{2} \,\dd x \\
=& \, - \int_{\mathbb{R}^d} |\nabla^{2} \uu |^{2} \varphi^{2} \, \dd x - 2 \sum_{i=1} ^n  \int_{\mathbb{R}^d} \varphi \nabla \varphi \otimes \nabla u_i \cdotdot \nabla^2 u_i \, \dd x 
- \sum_{i=1} ^n \int_{\R^d} \varphi^2 |\nabla u_i |^2 \frac{W'' (u_i)}{\varepsilon^2} \, \dd x\\
&+ \sum_{i=1} ^n \left\{ \int_{\R^d } \varphi ^2 \nabla u_i \nabla (g(\uu)) \frac{\sqrt{2W(u_i)}}{\varepsilon^{1+\alpha}} + \varphi ^2 |\nabla u_i|^2 g(\uu) \frac{W'(u_i)}{\sqrt{2W(u_i)}} \frac{1}{\varepsilon ^{1+\alpha}} \, \dd x\right\}\\
\leq& \, C \int_{\mathrm{supp} \,  \varphi} |\nabla \varphi| ^2 |\nabla \uu|^2 \, \dd x
- \frac{C}{\varepsilon^2} \int_{\R^d} \varphi ^2 |\nabla \uu|^2 \, \dd x \\
\leq& \, C\varepsilon ^{-2\beta'} \int_{\mathrm{supp} \,  \varphi} |\nabla \uu|^2 \, \dd x
- \frac{C}{\varepsilon^2} \int_{\R^d} \varphi ^2 |\nabla \uu|^2 \, \dd x
\end{split}
\end{equation}
for sufficiently small $\varepsilon$, where $g(\uu):=\Big(\sigma - \sum_{j=1}^n \phi(u_j)\Big)$. 
Here we used $W'' (u_i (x,s)) \geq C$,
\[
\sum_{i=1} ^n \left\{ \int_{\R^d } \varphi ^2 \partial_{x_k} u_i \partial_{x_k} (g(\uu)) \frac{\sqrt{2W(u_i)}}{\varepsilon^{1+\alpha}} + \varphi ^2 |\nabla u_i|^2 g(\uu) \frac{W'(u_i)}{\sqrt{2W(u_i)}} \frac{1}{\varepsilon ^{1+\alpha}} \, \dd x\right\}
\leq \frac{C}{\varepsilon ^{1+\alpha}} \int_{\R^d} \varphi^2 |\nabla \uu|^2 \, \dd x,
\]
and $|\nabla \varphi|\leq 2\varepsilon ^{-\beta'}$. Note that $\|u_i \|_\infty \leq 1$ and 
$\sup_{0\leq s \leq 1}\frac{|W'(s)|}{\sqrt{2W(s)}} <\infty$.
Set 
\[
F(t):=\frac12 \int_{\mathbb{R}^d} \varphi^{2} |\nabla \uu (x,t)|^{2} \, \dd x,
\qquad
M:=\sup\limits_{\tau \in [t_{\ast}-\varepsilon^{2\beta^{\prime}}, t_{\ast}]} \int_{\mathrm{supp} \,  \phi} \frac{1}{2} |\nabla \uu (x, \tau) |^{2} \, \dd x.
\]
Then we have
\[
\frac{\dd}{\dd t} F(t) \leq -\frac{C}{\varepsilon^2} F(t) + C\varepsilon^{-2\beta'} M
\]
and by integrating this over $[t_{\ast}-\varepsilon^{2\beta^{\prime}}, t_{\ast}]$, we obtain
\begin{equation}\label{437}
\begin{split}
F(t_\ast) \leq M (e^{-C\varepsilon^{2(\beta' -1)}} + C\varepsilon^{2(1-\beta')} ). 
\end{split}
\end{equation}
By $\mathrm{supp} \,  \phi \subset B_{2r}(y)$, it holds that
\[
\int_{\mathrm{supp} \,  \phi} \frac{\varepsilon}{2}|\nabla \uu|^{2} \, \dd x \le \mu_{t}^{\varepsilon}(B_{2r}(y)) \le \omega_{d-1}D_{1}(2r)^{d-1} 
\]
for all $t \in [t_{\ast}-\varepsilon^{2\beta^{\prime}}, t_{\ast}]$ and thus
$\varepsilon M \le \omega_{d-1}D_{1}(2r)^{d-1}$.
Hence, for sufficiently small $\varepsilon>0$,
\begin{equation*}
\begin{split}
& \int_{\R^d} \frac{\varepsilon}{2} |\nabla \uu |^{2} \varphi^{2} (x, t_{\ast}) \, \dd x 
=\varepsilon F (t_\ast)
\le \varepsilon M (e^{-C\varepsilon^{2(\beta' -1)}} + C\varepsilon^{2(1-\beta')} ) \\
\le & \, \omega_{d-1}D_{1}(2r)^{d-1} (e^{-C\varepsilon^{2(\beta' -1)}} + C\varepsilon^{2(1-\beta')} ) 
\le  
\omega_{d-1}D_{1}(2r)^{d-1} (2C\varepsilon^{2(1-\beta')} ) \\
= & \, 
\omega_{d-1}D_{1}(2r)^{d-1} (2C\varepsilon^{2(\beta'-\beta)} ) ,
\end{split}
\end{equation*}
where we used $0<\beta'<1$ and $1-\beta'=\beta' -\beta$. 
Since $B_{r} \backslash A \subset \{ \varphi =1 \}$, we have
\begin{equation}\label{439}
\int_{B_{r} (y) \backslash A} \frac{\varepsilon}{2} |\nabla \uu|^{2} \varphi^{2} (x, t_{\ast}) \, \dd x 
\le \int_{\R^d} \frac{\varepsilon}{2} |\nabla \uu|^{2} \varphi^{2} (x, t_{\ast}) \, \dd x \le C \varepsilon^{\beta^{\prime}-\beta}r^{d-1}.
\end{equation}
By \eqref{434} and \eqref{439}, we obtain
\begin{equation*}
\int_{B_{r}(y) }  ( \tilde \xi(x,t_\ast))_{+}  \, \dd x 
\le C \varepsilon^{\beta^{\prime}-\beta}r^{d-1}.
\end{equation*}
The proof is complete.
\end{proof}

Now we are able to improve upon the integral estimate from Theorem~\ref{thm:positive part of disrepancy measure}, by showing that even when integrating against the singular kernel $\frac1{s-t}\rho_{(y,s)}(x,t) \chi_{t\leq s}$ that appears on the right-hand side of the monotonicity formula in Lemma~\ref{lemma:monotonicity formula}, the positive part of the discrepancy measure is still small.
\begin{lemma}\label{lemma9}
There exists $C>0$ such that for any $y \in \mathbb{R}^d$, $0<\varepsilon < \varepsilon_{3}$, $t \in [0, T_{\varepsilon}]$, $t<s$, we have
\begin{equation}\label{442}
\int_{0}^{t} \left[ \frac{1}{2(s-\tau)}\int_{\R^d} (\tilde \xi(x,\tau ))_{+} \rho_{(y,s)} (x, \tau)  \, \dd x\right] \, \dd \tau \le C \varepsilon^{\beta^{\prime}-\beta} |\log \varepsilon|.
\end{equation}
Moreover, for any $t_1 <t_2 <s$ with $t_1, t_2 \in [0, T_{\varepsilon}]$ we have
\begin{equation}\label{MF2}
\begin{split}
& \int_{\R^d} \rho_{(y,s)} (x,t_2) \,\dd \omega^\eps_{t_2}(x)
  + \int_{t_1} ^{t_2}  \frac1{2(s-t)} \int_{\R^d} \rho_{(y,s)}(x,t) \,\dd \xi^\eps_{t,-} (x) \dd t \\
  \leq & \,\int_{\R^d} \rho_{(y,s)} (x,t_1) \,\dd \omega^\eps_{t_1}(x)
 + C \varepsilon^{\beta^{\prime}-\beta} |\log \varepsilon|.
\end{split}
\end{equation}
\end{lemma}
\begin{proof}
If $t \le 2\varepsilon^{2\beta^{\prime}}$, then, by \eqref{discrepancy1} and $\int_{\R^d} \rho_{(y,s)} (x,t) \, \dd x = \sqrt{4\pi (s-\tau)}$, we have
\begin{equation}\label{443}
\begin{split}
& \int_{0}^{t} \left[ \frac{1}{2(s-\tau)}\int_{\R^d} (\tilde \xi(x,\tau))_{+} \rho_{(y,s)} (x, \tau)  \, \dd x\right] \, \dd \tau \\
\le &\, \int_{0}^{t} \frac{C\varepsilon^{-\beta} \sqrt{\pi} }{\sqrt{s-\tau}} \dd\tau 
\le \int_{0}^{t} \frac{C\varepsilon^{-\beta} \sqrt{\pi} }{\sqrt{t-\tau}} \dd\tau
\le C \varepsilon^{\beta^{\prime}-\beta}.
\end{split}
\end{equation}
If $s>t\geq s-2\varepsilon^{2\beta^{\prime}}$, then by a similar argument, we have
\[
\int_{s-2\varepsilon^{2\beta^{\prime}}}^{t} \left[ \frac{1}{2(s-\tau)}\int_{\R^d } (\tilde \xi(x,\tau))_{+} \rho_{(y,s)} (x, \tau) \, \dd x\right]\dd\tau 
\le C \varepsilon^{\beta^{\prime}-\beta}.
\]
So we need to estimate on $[2\varepsilon^{2\beta^{\prime}}, t]$ with $t\le s-2\varepsilon^{2\beta^{\prime}}$.
On $B_{\varepsilon^{\beta^{\prime}}}(y)$, by \eqref{discrepancy1} and $s-t \geq 2\varepsilon^{2\beta^{\prime}}$, we have
\begin{equation}\label{444}
\begin{split}
&\int_{2\varepsilon^{2\beta^{\prime}}}^{t}  \left[ \frac{1}{2(s-\tau)}\int_{B_{\varepsilon^{\beta^{\prime}}}(y)} (\tilde \xi(x,\tau) )_{+} \rho_{(y,s)} (x, \tau) \,  \dd x\right]\dd \tau \\
&\le \int_{2\varepsilon^{2\beta^{\prime}}}^{t} \frac{C \omega_{d}\varepsilon^{\beta^{\prime}d-\beta}}{(\sqrt{4\pi})^{d-1}(s-\tau)^{\frac{d+1}{2}}} \dd\tau \le C \varepsilon^{\beta^{\prime}-\beta} .
\end{split}
\end{equation}
On $\R^d \backslash B_{\varepsilon^{\beta^{\prime}}}(y) $, by \eqref{431} and $\int_0 ^1 (-\log l)^{\frac{d-1}{2}} dl= \pi ^{\frac{d-1}{2}} \omega_{d-1} ^{-1}$, we obtain
\begin{equation}\label{445}
\begin{split}
&\int_{2\varepsilon^{2\beta^{\prime}}}^{t}  
\left[ \frac{1}{2(s-\tau)}\int_{\R^d \backslash B_{\varepsilon^{\beta^{\prime}}}(y)} (\tilde \xi(x,\tau))_{+} \rho_{(y,s)} (x, \tau) \, \dd x\right]\dd\tau \\
\le& \int_{2\varepsilon^{2\beta^{\prime}}}^{t} \frac{1}{2(s-\tau)^{\frac{d+1}{2}}(\sqrt{4\pi})^{d-1}} \int_{0}^{1} \left\{ \int_{\{ x | e^{-\frac{|x-y|^{2}}{4(s-\tau)}} \geq l \} \backslash B_{\varepsilon^{\beta^{\prime}}}(y) } (\tilde \xi(x,\tau))_{+} \right\} \, \dd l\, \dd\tau \\
=& \int_{2\varepsilon^{2\beta^{\prime}}}^{t} \frac{1}{2(s-\tau)^{\frac{d+1}{2}}(\sqrt{4\pi})^{d-1}} \int_{0}^{1} \left\{ \int_{\{\varepsilon ^{\beta'} \leq |x-y|\leq \sqrt{4(s-\tau ) (-\log l)} \} } (\tilde \xi (x,\tau))_{+}\right\} \, \dd l\, \dd\tau \\
\leq & \int_{2\varepsilon^{2\beta^{\prime}}}^{t} \frac{1}{2(s-\tau)^{\frac{d+1}{2}}(\sqrt{4\pi})^{d-1}} \int_{0}^{1} 
C \varepsilon ^{\beta' -\beta} 
(\sqrt{4(s-\tau) (-\log l)})^{d-1} \,
\dd l \,\dd\tau \\
\le& C\varepsilon^{\beta^{\prime}-\beta} \int_{2\varepsilon^{2\beta^{\prime}}}^{t} \frac{(s-\tau)^{\frac{d-1}{2}}}{(s-\tau)^{\frac{d+1}{2}}} \, \dd \tau 
\le C\varepsilon^{\beta^{\prime}-\beta} \beta^{\prime}|\log \varepsilon|,
\end{split}
\end{equation}
where we used $s-t\geq 2\varepsilon ^{2\beta'}$ and $s\geq 4\varepsilon ^{2\beta'}$, since $2\varepsilon ^{2\beta'} \leq t \leq s- 2\varepsilon ^{2\beta'}$.
Combining \eqref{443}, \eqref{444}, and \eqref{445}, we obtain the estimate \eqref{442}.
\end{proof}
Finally, we can prove the upper bound~\eqref{upperbound of density} on the density.
\begin{theorem}\label{thm:density}
For any $T>0$, there exists $\eps_{0}$ such that
\begin{equation}\label{eq:upperbound}
\inf_{\varepsilon \in (0,\eps_0)} T_\varepsilon \geq T.
\end{equation}
Hence
\begin{equation}
\sup_{\varepsilon \in (0,\eps_0)}\sup\limits_{t \in[0, T]} D_\varepsilon (t) \leq D_1. 
\end{equation}
\end{theorem}
\begin{proof}
We will prove the statement by contradiction. Let $T>0$ and suppose that there exists a sequence $\{ \varepsilon_{i} \}_{i=1}^{\infty}$ such that $\varepsilon_{i} \to 0$ and $T_{\varepsilon_i} <T $ holds for all $i \in \mathbb{N}$.
Note that $1< D_{1}=D_{\varepsilon_i}(T_{\varepsilon_i})$.
From the definition of $D_\varepsilon (t)$, there exist $y_i \in \mathbb{R}^d$ and $r_i \in (0, 1)$ such that 
\[
\frac{\omega_{T_{\varepsilon_i}}^{\varepsilon_{i}}(B_{r_i}(y_i))}{\omega_{d-1}r_i^{d-1}} > \frac{D_{\varepsilon_i} (T_{\varepsilon_i})}{2}=\frac{D_1}{2}.
\]
By Lemma~\ref{lemma:ilmanen}\eqref{item:ilmanen1}, we have
\begin{equation*}
\int_{\R^d} \rho_{(y_i,s_i)} (x, 0) \, \dd \omega_{0}^{\varepsilon_{i}} \le D_{0}.
\end{equation*}
Set $s_i=T_{\varepsilon_i} + r_i^{2}$.
By \eqref{442}, for $t \in [0, T_{\varepsilon_i}]$ and sufficiently small $\varepsilon_{i}$ we have
\begin{equation*}
\int_{0}^{t} \left[ \frac{1}{2(s_i-\tau)}\int_{\R^d} (\xi_{\tau}^{\varepsilon_{i}})_{+} \rho_{(y_i,s_i)} (x, \tau) \, \dd x\right]\dd\tau \le C \varepsilon_{i}^{\beta^{\prime}-\beta} |\log \varepsilon_{i}|.
\end{equation*}
We compute 
\begin{equation*}
\int_{\R^d} \rho_{(y_i,s_i)} (x,T_{\varepsilon_i}) \, \dd \omega_{T_{\varepsilon_i}} ^{\varepsilon_i} \geq \int_{B_{r_i}(y_i)} \rho_{(y_i,s_i)} (x,T_{\varepsilon_i}) \, \dd \omega_{T_{\varepsilon_i}}^{\varepsilon_{i}} \geq \frac{e^{-\frac{1}{4}}}{(4\pi)^{\frac{d-1}{2}}r_i^{d-1}} \omega_{T_{\varepsilon_i}}^{\varepsilon_{i}} (B_{r}(y)) >
\frac{e^{-\frac{1}{4}}\omega_{d-1}}{2(4\pi)^{\frac{d-1}{2}}}D_{1}.
\end{equation*}
By the monotonicity formula \eqref{MF2}, we have
\begin{equation*}
\begin{split}
 \frac{e^{-\frac{1}{4}}\omega_{d-1}}{2(4\pi)^{\frac{d-1}{2}}}D_{1}
&<
\int_{\R^d} \rho_{(y_i,s_i)} (x,T_{\varepsilon_i}) \,\dd \omega^{\eps_i} _{T_{\varepsilon_i}}(x) \\
  &\leq \int_{\R^d} \rho_{(y,s)} (x,0) \,\dd \omega^{\eps_i} _{0}(x)
 + C \varepsilon_i ^{\beta^{\prime}-\beta} |\log \varepsilon_i | \\
&  \leq  D_0
 + C \varepsilon_i ^{\beta^{\prime}-\beta} |\log \varepsilon_i | \leq 2D_0
\end{split}
\end{equation*}
for sufficient small $\varepsilon_i$.
Then we have
\[
\frac{e^{-\frac{1}{4}}\omega_{d-1}}{4(4\pi)^{\frac{d-1}{2}}}D_{1} \le D_{0},
\]
which is a contradiction. Hence we obtain \eqref{eq:upperbound}.
\end{proof}
By Lemma~\ref{lemma9} and Theorem~\ref{thm:density} we obtain the following energy-decay estimate from the monotonicity formula in Lemma~\ref{lemma:monotonicity formula}.
\begin{corollary}\label{cor:discrepancy decay}
There exists $C>0$ such that for any $y \in \mathbb{R}^d$, $0<\varepsilon < \varepsilon_{4}$, $t \in [0, T]$, $t<s$, we have
\begin{equation}\label{442_a}
\int_{0}^{t} \left[ \frac{1}{2(s-\tau)}\int_{\R^d} (\tilde \xi(x,\tau ))_{+} \rho_{(y,s)} (x, \tau)  \, \dd x\right] \, \dd \tau \le C \varepsilon^{\beta^{\prime}-\beta} |\log \varepsilon|.
\end{equation}
Moreover, for any $t_1 <t_2 <s$ with $t_1, t_2 \in [0, T]$ we have
\begin{equation}\label{MF2_a}
\begin{split}
& \int_{\R^d} \rho_{(y,s)} (x,t_2) \,\dd \omega^\eps_{t_2}(x)
  + \int_{t_1} ^{t_2}  \frac1{2(s-t)} \int_{\R^d} \rho_{(y,s)}(x,t) \,\dd \xi^\eps_{t,-} (x) \dd t \\
  \leq & \,\int_{\R^d} \rho_{(y,s)} (x,t_1) \,\dd \omega^\eps_{t_1}(x)
 + C \varepsilon^{\beta^{\prime}-\beta} |\log \varepsilon|.
\end{split}
\end{equation}
\end{corollary}

\begin{remark}
The inequality \eqref{442_a} cannot be derived directly from Theorem \ref{thm:positive part of disrepancy measure}, because it holds uniformly even as $s\downarrow t$.
\end{remark}

\bigskip

We now show that $\uu_\eps$ cannot be close to the wells around a point where the energy concentrates.
This is again in stark contrast to the case of the Ginzburg--Landau equation.

\begin{lemma}[Distance to wells]\label{lemma:distance to wells}
 For any $(x,t)\in \supp \omega$ with $t>0$ there exist an index $ i \in \{ 1,\dots, n \}$ and sequences $\eps_k\to0$ and $(x_k,t_k)\to (x,t)$ such that $\frac13 < u_{i,\varepsilon _k} (x_k,t_k) <\frac23$.
\end{lemma}

The proof closely follows the one in~\cite{MizunoTonegawa}.
\begin{proof}
  Assume that there exists $(x_0,t_0)\in \supp \omega$ with $t_0>0$, and $r_0>0$ such that for all $i\in\{1,\ldots,n\}$, there exists $\eps_0>0$ such that for all sufficiently large $k$, we have
  \[
    u_i^{\eps_k}(x,t) \notin \Big(\frac13,\frac23\Big) \quad \text{for all } (x,t)\in B_{r_0}(x)\times (t_0-r_0^2,t_0+r_0^2).
  \]
  We want to conclude from this that $(x_0,t_0)\notin \supp \omega$.

  Fix $i\in\{1,\ldots,n\}$ and $\eps<\eps_0$. 
  By continuity, for every such $k$, either $u_i^{\eps_k} \leq \frac13$ in $B_{r_0}(x)\times (t-r_0^2,t+r_0^2)$ or $u_i^{\eps_k} \geq \frac23$ in $B_{r_0}(x_0)\times (t_0-r_0^2,t_0+r_0^2)$. Passing to a subsequence, we may without loss of generality assume that $u_i^{\eps} \leq \frac13$ in $ B_{r_0}(x_0)\times (t_0-r_0^2,t_0+r_0^2)$. The other case is analogous.

  For notational convenience, we omit the index $k$ on $\eps$ and write $u_i $ for $u_i^\eps$. 
  Let $\varphi \in C_c^\infty(B_{r_0}(x_0))$ such that $|\nabla \varphi|\leq \frac 3{r_0}$ and $\varphi=1$ on $B_{r_0/2}(x_0)$. 
  Testing the $i$-th component of the Allen--Cahn equation~\eqref{eq:allen cahn} with $\varphi^2 u_i^{\eps}$, denoting $g=\sigma-\sum_{i} \phi(u_i)$, we obtain
  \begin{align*}
    \frac{\dd}{\dd t} \int_{\torus}\varphi^2 \frac\eps2 u_i^2 \,\dd x 
    &= 
    \int_{\torus} \Big(  \eps \varphi^2  u_i \Delta u_i -  \varphi^2 u_i \partial_{u_i} F_\eps(\uu_\eps)  \Big)\,\dd x
    \\&
    \int_{\torus} \Big( - \eps \varphi^2  |\nabla u_i|^2 - 2 \eps \varphi u_i \nabla \varphi \cdot \nabla u_i -  \varphi^2 u_i \frac1\eps W'(u_i) + \varphi^2 u_i g \frac1{\eps^\alpha } \sqrt{2W(u_i)}   \Big)\,\dd x.
  \end{align*}
  By Cauchy--Schwarz, $2|\varphi u_i \nabla \varphi \cdot \nabla u_i| \leq \varphi^2 |\nabla u_i|^2 + u_i^2 |\nabla \varphi|^2$. Moreover, since $u_i \leq \frac13$, there exists a constant $C<\infty$ such that $u_i W'(u_i) \geq \frac1C W(u_i)$ and $u_i\sqrt{2W(u_i)} \leq C W(u_i)$. 
  Hence
   \begin{align*}
    \frac{\dd}{\dd t} \int_{\torus}\varphi^2 \frac\eps2 u_i^2 \,\dd x 
    &\leq  
    \int_{\torus} \eps u_i^2 |\nabla \varphi|^2 \, \dd x -  \big(\frac1C -C \|g\|_\infty \eps^{1-\alpha}\big) \int_{\torus} \varphi^2 \frac1\eps W(u_i) \,\dd x.
   \end{align*}
   Integrating in time and reordering yields
   \begin{align*}
    &\big(\frac1C -C \|g\|_\infty \eps^{1-\alpha}\big) \int_{t_0-r_0^2}^{t_0+r_0^2}\int_{\torus} \varphi^2 \frac1\eps W(u_i) \,\dd x \, \dd t
    \\&\leq 
    -\int_{\torus}\varphi^2 \frac\eps2 u_i^2 \,\dd x\Big|_{t=t_0-r_0^2}^{t=t_0+r_0^2}
    +\int_{t_0-r_0^2}^{t_0+r_0^2}\int_{\torus} \eps u_i^2 |\nabla \varphi|^2 \, \dd x \, \dd t.
   \end{align*}
   As the right-hand side vanishes in the limit $\eps\to0$ due to the uniform bounds~\eqref{gradient-est}, we obtain
   \begin{align}\label{eq:pf_Wto0}
       \lim_{\eps\to0}\int_{t_0-r_0^2}^{t_0+r_0^2}\int_{\torus} \varphi^2 \frac1\eps W(u_i) \,\dd x \, \dd t =0.
   \end{align}
   
   Moreover,
   \begin{align*}
    \int_{t_0-r_0^2}^{t_0+r_0^2}\int_{\torus} \varphi^2 \frac\eps2 |\nabla u_i|^2 \,\dd x \, \dd t
    &\leq
    \int_{t_0-r_0^2}^{t_0+r_0^2}\int_{\torus} \varphi^2 (\xi_\eps)_+ \,\dd x \, \dd t
    + \int_{t_0-r_0^2}^{t_0+r_0^2}\int_{\torus} \varphi^2 \frac1\eps W(u_i)\,\dd x \, \dd t.
   \end{align*}
   By Theorem~\ref{thm:chen}, the first right-hand side term vanishes in the limit $\eps\to0$ and hence, by~\eqref{eq:pf_Wto0}, we get
   \begin{align}\label{eq:pf_Duto0}
       \lim_{\eps\to0}\int_{t_0-r_0^2}^{t_0+r_0^2}\int_{\torus} \varphi^2 \frac\eps2 |\nabla u_i|^2 \,\dd x \, \dd t =0.
   \end{align}

   Hence, by~\eqref{eq:pf_Wto0} and~\eqref{eq:pf_Duto0}, we have $(x_0,t_0)\notin \supp \omega$, which is precisely what we wanted to prove.
\end{proof}

Now the following lemmas follow with proofs adapted from~\cite{MizunoTonegawa}, as they mostly rely on the monotonicity formula and Lemma~\ref{lemma:distance to wells}.

\begin{lemma}[Clearing out lemma]\label{lemma:clearing out}
  There exist $\gamma_0,\delta_0>0$ depending only on $T>0$ and $W$ such that the following holds: 
  If
  \begin{align}
    \int_{\R^d} \rho_{(y,s)}(x,t) \,\dd \omega_s(y) <\delta_0
  \end{align}
  for some $x\in \R^d$ and some $0\leq t<s<T/2$, then $(x',t')\notin \supp \omega$ for all $x'\in B_{\gamma_0 r} (x)$, where $t':=2s-t$ and $r:=\sqrt{2(s-t)}$.
\end{lemma}

\begin{proof}
Assume $(x' , t') \in \supp \mu$ for a conrtadiction. By this assumption, Lemma \ref{lemma:distance to wells} implies that there exist a subsequence $\varepsilon _k \to 0$, $i \in \{1,\dots,n\}$, and a sequence $\{ (x_k,t_k) \}_{k=1} ^\infty$ with 
$(x_k ,t_k) \to (x', t')$ as $k \to \infty$ and 
\begin{equation}\label{col-1}
\frac14 < u_{i,\varepsilon_k} (x_k ,t_k) < \frac34 \qquad \text{for any} \ k\geq 1. 
\end{equation}
Set $T_k := t_k + \varepsilon _k ^2$. Similar to that for the calculation of \eqref{422}, by \eqref{col-1} there exists a constant $c_0 >0$ such that for any $k\geq1$
\begin{equation*}
c_0 \leq \int_{\R^d} \rho _{(x_k , T_k)} (y,t_k) \, \dd \mu_{t_k} ^{\varepsilon _k} (y). 
\end{equation*}
By this and \eqref{MF2_a}, 
\begin{equation*}
\begin{split}
 c_0 
&\leq \int_{\R^d} \rho_{(x_k,T_k)} (y, t_k) \,\dd \omega^{\eps_k}_{t_k}(y)
\\&\leq \int_{\R^d} \rho_{(x_k ,T_k)} (y,s) \,\dd \omega^{\eps_k}_{s}(y)
 + C \varepsilon_k ^{\beta^{\prime}-\beta} |\log \varepsilon_k |
\end{split}
\end{equation*}
for any $k \geq 1$. Letting $k\to \infty$, we have
\begin{equation}\label{col-2}
\begin{split}
c_0 \leq \,\int_{\R^d} \rho_{(x' , t' )} (y,s) \,\dd \omega_{s}(y).
\end{split}
\end{equation}
By $r=\sqrt{2(s-t)}=\sqrt{2(t' -s)}$, \eqref{col-2}, and Lemma~\ref{lemma:ilmanen}\eqref{item:ilmanen3}, 
for any $\delta>0$, there exists $\gamma (\delta)>0$ such that for any $x \in B_{\gamma (\delta) r} (x')$,
\begin{equation*}
\begin{split}
c_0 \leq & \,\int_{\R^d} \rho_{(x' , t' )} (y,s) \,\dd \omega_{s}(y) \\
\leq & \, (1+\delta) \int_{\R^d} \rho_{(x , t' )} (y,s) \,\dd \omega_{s}(y) + \delta D(s), 
\end{split}
\end{equation*}
where $D(s):= \sup_{z \in \R^d, R >0} \frac{\omega_s (B_R(z))}{\omega_{d-1} R^{d-1}}$. By \eqref{upperbound of density}, $\sup_{s \in [0,T]} D(s) \leq D_0$ for some $D_0>0$.
Thus, we may choose $\delta>0$ such that
\begin{equation*}
\begin{split}
\frac{c_0}{2} \leq \int_{\R^d} \rho_{(x , t' )} (y,s) \,\dd \omega_{s}(y),
\end{split}
\end{equation*}
and the claim holds with $\delta_0 :=\frac{c_0}{2}$ and $\gamma_0:=\gamma (\delta_0)$.
\end{proof}

Since the following lemma relies only on Lemma \ref{lemma:clearing out} and the upper bound of the density \eqref{upperbound of density}, its proof is almost identical to that of \cite{Ilmanen_allencahn,MizunoTonegawa,TakasaoTonegawa_transport} and is therefore omitted.

\begin{lemma}[Lower bounds on forward Gaussian density]\label{Lower bounds on forward Gaussian density}
  For $T>0$ let $\delta_0(T)>0$ be the constant from Lemma~\ref{lemma:clearing out} and let
  \[
    Z^-(T) :=
    \Big\{
      (x,t)\in \supp \omega 
      \,\Big|\,
      0\leq t <\frac{T}{2} \text{ and }
      \limsup_{s\downarrow t} \int_{\mathbb{T}^d} \rho_{(y,s)}(x,t)\,\dd \omega_s(y) <\delta_0(T)
    \Big\}
  \]
  satisfies $\omega(Z^-(T)) =0$.
\end{lemma}

Now we integrate the monotonicity formula from Lemma~\ref{lemma:monotonicity formula}, with the bound from Theorem~\ref{thm:positive part of disrepancy measure}, over the base point $y$ and the time horizon $s$ to prove Theorem~\ref{thm:discrepancy measure full}.

\begin{proof}[Proof of Theorem~\ref{thm:discrepancy measure full}]
  This now, too, follows as in~\cite{Ilmanen_allencahn,MizunoTonegawa}.
However, for the convenience of the reader, we briefly outline the strategy.
By \eqref{eq:integral discprepancy positive vanishes}, we only need to prove
  \[
    \dd \xi^\eps_- = \Big( \frac\eps2 |\nabla \uu_\eps|^2 - F_\eps(\uu_\eps)\Big)_- \dd x \,\dd t \stackrel{\ast}{\rightharpoonup} 0 \quad\text{as }\eps\to0.
  \]
We may assume that there exists a Radon measure $\xi_{-}$ such that $\xi_ - ^\varepsilon \stackrel{\ast}{\rightharpoonup} \xi_-$ by Banach--Alao\u{g}lu. 
By \eqref{MF2_a}, letting $\varepsilon \to 0$, we have
\begin{equation}\label{proof Thm3-1}
\begin{split}
  \int_{\R^d\times(0,s)} \frac{\rho_{(y,s)}(x,t)}{2(s-t)} \,\dd \xi_{-} (x,t) 
  \leq C<\infty
\end{split}
\end{equation}
holds some $C>0$. 
By Fubini's theorem and \eqref{proof Thm3-1}, we have
\begin{equation}\label{proof Thm3-2}
\begin{split}
& \int_{\R^d\times(0,s)} \left(\int_{0} ^ T \int_{\mathbb{T}^d} \frac{\rho_{(y,s)}(x,t)}{2(s-t)} \, \dd \omega_s (y) \dd s \right) \dd \xi_{-} (x,t) \\
= & \, \int_{0} ^ T \int_{\mathbb{T}^d} \left( \int_{\R^d\times(0,s)} \frac{\rho_{(y,s)}(x,t)}{2(s-t)} \,\dd \xi_{-} (x,t)\right) \dd \omega_s (y) \dd s
\leq C T \omega_ 0 (\mathbb{T}^d) <\infty,
\end{split}
\end{equation}
where we used the monotonicity of $\omega_ s (\mathbb{T}^d)$. 
Then, by \eqref{proof Thm3-2}, similar to that for the argument in~\cite{Ilmanen_allencahn,MizunoTonegawa} one can prove that for $\xi_-$-almost all $(x,t)$, 
\[
\lim_{s \downarrow t} \int_{\mathbb{T}^d} \rho _{y,s} (x,t) \, \dd \omega_s (y)=0,
\]
and thus $\xi_- (\mathbb{T}^d \times (0,T) \setminus Z^- (T))=0$. Therefore,
\[
\xi_- (\mathbb{T}^d \times (0,T) ) \leq \xi_- (Z^- (T)) \leq \omega (Z^- (T))=0,
\]
where we used Lemma \ref{Lower bounds on forward Gaussian density} and $\xi_- (A) \leq \omega (A)$ for any Borel set $A$, which follows immediately from the definition. This completes the proof of the theorem.
\end{proof}
The following lemma yields a lower bound on the density, which is essential to prove the rectifiablity of the associated varifolds.

\begin{lemma}\label{lemma:lower density bounds}
\begin{enumerate}[(i)]
\item \label{item1} There exists $C(T)>0$ such that 
\begin{equation}\label{supp_omega_t}
\mathcal{H}^{d-1} (\supp \omega_t) \leq C(T) \liminf_{r\downarrow 0} \omega_{t-r^2} (\mathbb{T}^d)
\end{equation}
holds for any $0 < t \leq T$.
\item \label{item2} For any $0<t <\infty$, we have
\begin{equation}\label{density_zero_sets}
\omega_t \left(\left\{ x \in \supp \omega_t \ \middle| \ \lim_{r\downarrow 0} \frac{\omega_t (B_r (x))}{\omega_{d-1} r^{d-1}}=0 \right\} \right)=0.
\end{equation}
\end{enumerate}
\end{lemma}
\begin{proof}
\step{Argument for~(\ref{item1}).} By following virtually the same argument as in \cite[Corollary~6.1]{TakasaoTonegawa_transport}, Lemma \ref{lemma:clearing out} yields that there exists $C(T)>0$ such that
\begin{equation*}
\mathcal{H}^{d-1} ((\supp \omega)_t) \leq C(T) \liminf_{r\downarrow 0} \omega_{t-r^2} (\mathbb{T}^d)
\end{equation*}
for any $0<t \leq T$, where 
$(\supp \omega)_t = \{ x \in \mathbb{T}^d \mid (x,t) \in \supp \omega \}$.
By the same proof as in~\cite[Lemma 5.1]{TakasaoTonegawa_transport}, $\supp \omega_t \subset (\supp \omega)_t$ holds. 
Hence we obtain \eqref{supp_omega_t}. 

\noindent
\step{Argument for~(\ref{item2}).} Fix an arbitrary $0 < t < \infty$. By \cite[Theorem 3.2]{LeonSimon} we have
\[
\omega _t \left( \left\{ x\in \supp \omega_t \ \middle| \ \limsup_{r\downarrow 0} \frac{\omega_t (B_r (x))}{\omega_{d-1} r^{d-1}} \leq s \right\}\right) 
\leq 2^{d-1} s \mathcal{H}^{d-1} (\supp \omega_t)
\]
for any $s>0$. By this and (\ref{item1}) we obtain \eqref{item2}.
\end{proof}

\section{Proof of Theorem~\ref{thm:ACtoBrakke}: The convergence proof}
\label{sec:convergence}
We first introduce the varifolds $\mu^\eps_t$ associated to our energy measures $\omega^\eps_t$.

\begin{definition}\label{def:varifold}
  For any $t>0$ and $\eps>0$, we define the varifold
  \begin{align*}
    \mu^\eps_t(\dd x,\dd P) := \sum_{i=1}^n \Big(\eps |\nabla u_{\eps,i}|^2 \,\dd x \otimes \delta_{I-\nu^\eps_i\otimes \nu^\eps_i}(\dd P)\Big).
  \end{align*} 
  In other words, for any test function on the Grassmannian bundle $\zeta \in C(\torus\times G_{d-1}(d))$, it holds
  \begin{align*}
      \int_{\torus\times G_{d-1}(d)} \zeta(x,P) \, \dd \mu^\eps_t(x,P)
      = \sum_{i=1}^n 
      \int_{\torus} \zeta(x, I-\nu^\eps_i\otimes \nu^\eps_i)  \eps|\nabla u_{\eps,i}|^2 \, \dd x.
  \end{align*}
\end{definition}

We have the following compactness for these varifolds.
\begin{proposition}\label{prop:compactness of varifolds}
  Let $\eps\downarrow0$ be a subsequence from Proposition~\ref{prop:compactness of energy measures}.
  Then, for almost every $t>0$, the limit $\omega_t$ is $(d-1)$-rectifiable, and there exists a further subsequence $\eps'(t)=\eps'\to0$ such that
  \[
    \mu^{\eps'}_t \stackrel{\ast}{\rightharpoonup} \mu_t,\quad\text{weakly-$\ast$ as measures on }\torus\times G_{d-1}(d),
  \]
  where $\mu_t$ is the unique rectifiable varifold associated to $\omega_t$. 
  Moreover, $\delta \mu_t = - \vec H_t \omega_t$ with
  \begin{align}\label{eq:liminfH}
    \int_{\torus} |\vec H_t|^2 \, \dd \omega_t \leq \liminf_{\eps'\to0} \int_{\torus} \frac1{\eps'} \big|\eps \Delta \uu_{\eps'}  -\partial_{\uu} F_{\eps'}(\uu_{\eps'}) \big|^2 \,\big|_{t} \dd x .
  \end{align}
\end{proposition}
Note carefully, that in contrast to Proposition~\ref{prop:compactness of energy measures}, here the subsequence depends on $t$. 
Nevertheless, this turns out to be enough as we can identify the limit. (Note that since $(0,\infty)$ is uncountable, this of course does not allow for a diagonal argument to select a single subsequence.)
\begin{proof}
  Fix $t\in(0,\infty)$. 
  First, by Banach--Alao\u{g}lu, since $\mu_t^\eps(\torus \times G_{d-1}(d)) = \omega^\eps_t(\torus) \leq \omega^\eps_0(\torus)\leq C$, we can extract a subsequence $\eps'\to0$ such that $\mu^{\eps'}_t \stackrel{\ast}{\rightharpoonup} \mu_t$ for some varifold $\mu_t$. 
  By Theorem~\ref{thm:discrepancy measure full} and Proposition~\ref{prop:compactness of energy measures}, $|\mu_t|=\omega_t$ for a.e.\ $t\in (0,\infty)$.
  Moreover, also the first variations converge, i.e., $\delta \mu^{\eps'}_t\stackrel{\ast}{\rightharpoonup}  \delta \mu_t$.
  Hence, for any vector field $X\in C^1(\torus;\R^d)$, we have
  \begin{align*}
    \delta \mu_t (X ) -\mu_t(|X|^2)
    &= \lim_{\eps'\to0} \Big(\delta \mu^{\eps'}_t(X) -\mu^{\eps'}_t(|X|^2) \Big).
  \end{align*}
  The second right-hand side term can be rewritten via
  \begin{align*}
   \mu^{\eps}_t(|X|^2) = \omega^\eps_t(|X|^2) +\xi^\eps_t(|X|^2) \to \omega_t(|X|^2) \quad \text{as }\eps\to0,
  \end{align*}
  and for the first one we have
  \begin{align*}
    \delta \mu^{\eps}_t(X) 
    &= \int_{\torus} \sum_{i=1}^n \nabla X \cdotdot (I-\nu^{\eps}_i \otimes \nu^{\eps}_i) \, \eps |\nabla u_{\eps,i}|^2 \, \dd x
    \\&=
     \int_{\torus} \nabla \cdot X \omega^\eps_t -
      \nabla X \cdotdot  \big( \eps 
      \nabla 
      \uu_{\eps}^\ast \nabla \uu_{\eps} \big)\,\dd x
    - \int_{\torus} \nabla \cdot X \, \dd \xi^\eps_t.
  \end{align*}
  Integrating by parts and using~\eqref{eq:IBPLuckhausModica} for the first integral, while estimating the second one via Theorem~\ref{thm:discrepancy measure full}, this yields
  \begin{align}\label{eq:deltamueps}
   \limsup_{\eps\to0}\bigg|
    \delta \mu^{\eps}_t(X)  
   - 
   \int_{\torus} X \cdot (\nabla \uu_\eps)^\ast (\eps \Delta \uu_\eps  -\partial_{\uu} F_\eps(\uu_\eps) ) \big|_{t}\,\dd x
   \bigg|=0.
  \end{align}
  Hence, evaluating along the subsequence $\eps'(t)$, by Cauchy--Schwarz and Theorem~\ref{thm:discrepancy measure full}, we obtain
  \begin{align*}
    |\delta \mu_t(X)| 
    &= \lim_{\eps'\to0} |\delta \mu^{\eps'}_t(X)|
    \\&\leq 
    \liminf_{\eps'\to0} \Big( \int_{\torus} |X|^2 \eps' |\nabla \uu_{\eps'}|^2 \big|_{t}\, \dd x\Big)^\frac12
    \Big( \int_{\torus} \frac1{\eps'} \big|\eps' \Delta \uu_{\eps'}  -\partial_{\uu} F_{\eps'}(\uu_{\eps'}) \big|^2\big|_{t} \, \dd x\Big)^\frac12
    \\&
    = \Big(\int_{\torus} |X|^2 \, \dd \omega_t\Big)^\frac12
    \Big( \liminf_{\eps'\to0} \int_{\torus} \frac1{\eps'} \big|\eps' \Delta \uu_{\eps'}  -\partial_{\uu} F_{\eps'}(\uu_{\eps'}) \big|^2 \,\big|_{t} \dd x \Big)^\frac12.
  \end{align*}
  By Fatou's lemma, the energy estimate~\eqref{eq:EDI} implies that for a.e.\ $t\in (0,\infty)$,
  \begin{align}\label{L2bound_for_H}
    \Big( \liminf_{\eps'\to0} \int_{\torus} \frac1{\eps'} \big| \eps' \Delta \uu_{\eps'}  -\partial_{\uu} F_{\eps'}(\uu_{\eps'}) \big|^2\big|_{t} \, \dd x \Big)^\frac12 <\infty,
  \end{align}
  and hence, any such limit $\mu_t$ has bounded first variation.
  Together with the lower density bounds from Lemma~\ref{lemma:lower density bounds}\eqref{item2}, this allows us to apply Allard's rectifiability theorem. Hence $\omega_t$ is rectifiable for a.e.\ $t\in (0,\infty)$.
  In that case the rectifiable varifold associated to $\omega_t$ is uniquely given by $\mu_t = \omega_t\otimes \delta_{T_x{\omega_t}}$, where $T_x{\omega_t}$ is the approximate tangent plane at $x$. Hence any converging subsequence converges to this varifold and we have~\eqref{eq:liminfH}.
\end{proof}

It is automatic that with the varifolds $\mu^\eps_t$, also their first variations $\delta\mu^\eps_t$ converge to the first variation $\delta \mu_t$ of the limit $\mu_t$.
Applying this to the transport term $\delta \mu^\eps_t(\nabla \varphi)$ in Brakke's inequality, we would get in the limit $\delta \mu_t (\nabla \varphi) = -\int_{\torus} \nabla \varphi \cdot \vec H_t \, \dd \omega_t$. 
To guarantee that this term has geometric meaning, we show that $\vec{H}_t$ is in fact a normal vector. 
This is done in the following proposition by adapting Ilmanen's ``Seven Epsilon Proof'' to our vectorial setting, cf.~\cite[pp.\ 450--451]{Ilmanen1994}.

\begin{proposition}[Brakke's transport term]\label{prop:7epsilon}
  For almost every $t\in (0,\infty)$, the limit varifold $\mu_t$ satisfies
  \begin{align}
    \delta \mu_t(X) 
    = -\int_{\torus} X \cdot (T_{(\cdot)} \omega_t)^\perp \vec H_t \, \dd \omega_t
  \end{align}
  for any $X \in C(\torus;\R^d)$. In other words, $\vec{H}_t$ is a normal vector $\omega_t$-a.e.
\end{proposition}

\begin{proof}
For a.e. $t\geq 0$, there exists a subsequence $\varepsilon'(t) \to 0$
such that $\mu_t ^{\varepsilon'} \overset{*}{\rightharpoonup} \mu_t$, 
\eqref{eq:liminfH}, and \eqref{L2bound_for_H} by Proposition \ref{prop:compactness of varifolds}. 
Fix such $t\geq 0$ and $\varepsilon' (t) \to 0$. Fix an arbitrary $\delta>0$. Since 
$ (T_{(\cdot)} \omega_t)^\perp X  \in L^2 (\omega_t ; \mathbb{R}^d)$,
there exists $Y \in C^1 (\mathbb{T}^d ; \mathbb{R}^d)$ such that 
\begin{equation}\label{7eps-2}
\begin{split}
\int_{\mathbb{T}^d} |Y -  (T_{(\cdot)} \omega_t)^\perp X  |^2 \, \dd \omega_t
\leq \delta ^2 .
\end{split}
\end{equation}
We decompose the desired bilinear term into six terms:
\begin{equation}\label{7eps-3}
\begin{split}
& \int_{\mathbb{T}^d} X \cdot (T_{(\cdot)} \omega_t)^\perp \vec{H}_t \, \dd \omega_t \\
= & \, \int_{\mathbb{T}^d} (X \cdot (T_{(\cdot)} \omega_t)^\perp  \vec{H}_t 
- Y\cdot \vec{H}_t) \, \dd \omega_t 
\\& + (\delta \mu_t (Y) - \delta \mu_t ^{\varepsilon'} (Y) ) \\
& + \left( \delta \mu_t ^{\varepsilon'} (Y) - \int_{\mathbb{T}^d} Y\cdot (\nabla \uu _{\varepsilon'}) ^{\ast} (\varepsilon' \Delta \uu_{\varepsilon'} -\partial _{\uu} F_{\varepsilon'} (\uu_{\varepsilon'})) \, \dd x \right) \\
& + \sum_{i=1} ^n \int_{\mathbb{T}^d} 
(Y \cdot \nabla u_{\varepsilon' ,i} - (X \cdot \nu_i ^{\varepsilon'}) \nu_i ^{\varepsilon'} \cdot \nabla u_{\varepsilon',i} )
(\varepsilon' \Delta u_{\varepsilon' ,i} -\partial _{u_i} F_{\varepsilon'} (\uu_{\varepsilon'})) \, \dd x \\
& + \sum_{i=1} ^n \int_{\mathbb{T}^d} \{ (X \cdot \nu_i ^{\varepsilon'}) \nu_i ^{\varepsilon'} \cdot \nabla u_{\varepsilon',i} -X \cdot \nabla u_{\varepsilon',i} \}(\varepsilon' \Delta u_{\varepsilon' ,i} -\partial _{u_i} F_{\varepsilon'} (\uu_{\varepsilon'})) \, \dd x \\
& + \sum_{i=1} ^n \int_{\mathbb{T}^d} X \cdot \nabla u_{\varepsilon',i} (\varepsilon' \Delta u_{\varepsilon' ,i} -\partial _{u_i} F_{\varepsilon'} (\uu_{\varepsilon'})) \, \dd x\\
=:& \, A+B+C+D+E+F.
\end{split}
\end{equation}
By \eqref{eq:liminfH}, \eqref{L2bound_for_H}, \eqref{7eps-2}, and the Cauchy--Schwarz inequality, there exists $C>0$ such that
\[
|A| \leq  C \delta.
\]
The convergence $\mu_t ^{\varepsilon'} \overset{*}{\rightharpoonup} \mu_t$ implies
$|B|\to 0$ as $\varepsilon'\to 0$.
By \eqref{eq:deltamueps}, we have $|C| \to 0$ as $\varepsilon'\to 0$.
By the Cauchy--Schwarz inequality and \eqref{L2bound_for_H}, 
\[
|D| \leq C \left( \sum_{i=1} ^n \int_{\mathbb{T}^d} |Y - (X\cdot \nu_i ^{\varepsilon'}) \nu_i ^{\varepsilon'}|^2 \varepsilon' |\nabla u_{\varepsilon',i}|^2 \, \dd x \right)^{\frac12}
\]
holds for some $C>0$. Set $\zeta =\zeta (x,P):= |Y (x) + X (x) \cdot (-I +P)|^2$. Then $\zeta (x, I-\nu_i ^{\varepsilon'} \otimes \nu_i ^{\varepsilon'}) =|Y - (X \cdot \nu_i ^{\varepsilon'}) \nu_i ^{\varepsilon'}|^2$ and thus 
\begin{equation*}
\begin{split}
|D|\leq & \, C \left( \sum_{i=1} ^n \int_{\mathbb{T}^d} |Y - (X \cdot \nu_i ^{\varepsilon'}) \nu_i ^{\varepsilon'}|^2 \varepsilon' |\nabla u_{\varepsilon',i}|^2 \, \dd x \right)^{\frac12}
=  C \left( \int_{\mathbb{T}^d \times G_{d-1} (d)} \zeta (x,P) \, \dd \mu_t ^{\varepsilon'} (x,P) \right)^{\frac12} \\
\to & \, C \left( \int _{\mathbb{T}^d \times G_{d-1} (d)} \zeta (x,P) \, \dd \mu_t (x,P) \right)^{\frac12}
\leq C\delta,
\end{split}
\end{equation*}
where we used \eqref{7eps-2} and $\zeta (x,P)= |Y - X\cdot (T_{(\cdot)} \omega_t)^\perp |^2$.
One can easily see $E=0$ since $(X \cdot \nu_i ^{\varepsilon'}) \nu_i ^{\varepsilon'} \cdot \nabla u_{\varepsilon',i} -X\cdot \nabla u_{\varepsilon',i}=0$. 
Finally, we have $F\to \delta \mu_t (X) $ as $\varepsilon'\to0$, since 
\eqref{eq:deltamueps} implies $|\delta \mu_t ^{\varepsilon'} (X) -F| \to 0$
and $|\delta \mu_t (X) -\delta \mu_t ^{\varepsilon'} (X)| \to 0$ by $\mu_t ^{\varepsilon'} \overset{*}{\rightharpoonup} \mu_t$.
Since $\delta > 0$ is arbitrary, we obtain the desired claim.
\end{proof}

While the measurability of the mass measures $\omega_t$ in time is immediate, the measurable dependence of the associated varifolds $\mu_t$ on $t$ is more subtle.
We argue similarly to~\cite[Theorem 3.10]{LiuWorkman}, to obtain this measurability.

\begin{proposition}\label{prop:measurability}
  The following functions, defined on $(0,\infty)$, are measurable:
\begin{enumerate}[(i)]
\item \label{measurability1} $t \mapsto \omega _t (\varphi)$, for any $\varphi \in C (\mathbb{T}^d)$,
\item \label{measurability2} $t \mapsto \mu_t (\zeta)$, for any $\zeta \in C(\mathbb{T}^d\times G_{d-1}(d))$,
\item \label{measurability3}
$t \mapsto \int_{\torus} X\cdot  \vec{H}_t\, \dd \omega_t$, for any $X \in C (\mathbb{T}^d;\R^d)$, and
\item \label{measurability5}
$t \mapsto \int_{\torus} \varphi |\vec{H}_t|^2 \, \dd \omega_t$, for any $\varphi \in C (\mathbb{T}^d)$.
\end{enumerate}
\end{proposition}
\begin{proof}
The proof is almost identical to the appendix in \cite{ChiesaTakasao} (see also \cite[Theorem 3.10]{LiuWorkman}). However, since the integrality of our family of varifolds $\{\mu_t\}_{t\geq 0}$ remains open, we prove only the parts where the arguments differ. Specifically, to establish (\ref{measurability2}) and (\ref{measurability5}) under our assumptions, it suffices to show the following:
If $s \geq 0$ and a sequence $t_j \to s$ with $t_j\geq0$ satisfy 
\begin{enumerate}[(a)]
\item $\mu_{t_j}$ and $\mu_s$ are rectifiable with $\sup_{j\geq 1} \int |\vec{H}_{t_j} |^2 \, \dd \omega_{t_j} <\infty$, $\int |\vec{H}_s |^2 \, \dd \omega_s <\infty$, 
\item $\omega_{t_{j}} \stackrel{\ast}{\rightharpoonup} \omega_s$,
\item $\mu_{t_{j}} \stackrel{\ast}{\rightharpoonup} \tilde \mu$ for some $\tilde \mu$, 
\end{enumerate}
then $\tilde \mu = \mu_s$. Since the first variations $\delta \mu_{t_j}$ are uniformly bounded, Allard's rectifiability theorem yields that
$\tilde \mu \llcorner \{ x \mid \theta ^{\ast (d-1)} (\| \tilde \mu \|,x) > 0 \} $ is rectifiable. 
Here $\theta ^{\ast (d-1)}(\| \tilde \mu \|,x) = \limsup_{r\to 0} \frac{\|\tilde \mu\|(B_r(x))}{r^{d-1}}$ denotes the upper $(d-1)$-dimensional density. Note that $\| \mu_{t_{j}} \| =\omega_{t_{j}}$, $\omega_{t_{j}} \stackrel{\ast}{\rightharpoonup} \omega_s$, and  $\mu_{t_{j}} \stackrel{\ast}{\rightharpoonup} \tilde \mu$ imply $\| \tilde \mu \| =\omega_s$.
By this and \eqref{density_zero_sets} we have
\[
\tilde \mu \llcorner \{ x \mid \theta ^{\ast (d-1)} (\| \tilde \mu \|,x) > 0 \}  = \tilde \mu \llcorner \{ x \mid \theta ^{\ast (d-1)} (\omega_s,x) > 0 \} = \tilde \mu
\] and thus $\tilde \mu$ is rectifiable.
Since 
$\|\tilde \mu \|=\omega_s = \| \mu_s \|$
and there is only one rectifiable varifold $\tilde \mu$ satisfying $\omega_s = \| \tilde \mu \|$ (this uniqueness holds only for rectifiable varifolds), we have $\tilde \mu=\mu_s$. 
\end{proof}

Now we are in the position to prove our main result.

\begin{proof}[Proof of Theorem~\ref{thm:ACtoBrakke}]
  By Propositions~\ref{prop:compactness of energy measures}, \ref{prop:compactness of u}, \ref{prop:compactness of varifolds}, and \ref{prop:measurability}, we have compactness and measurability in time, as well as the rectifiability of $\omega_t$ for a.e.\ $t$. We only need to show that any limit $\omega_t$ satisfies Brakke's inequality~\eqref{eq:Brakke}. 
  By a standard approximation argument, it is sufficient to prove it for test functions $\varphi\in C^2(\torus\times [0,\infty))$.

  To that end, let $0\leq t_1<t_2$ and $\varphi\in C^2(\torus\times [0,\infty))$ be fixed.
  We integrate the $\eps$-version of Brakke's inequality from Lemma~\ref{lemma:epsBrakke}, with a time-dependent test function $\varphi$ over $t$ from $t_1$ to $t_2$ to obtain
  \begin{align}
    \notag
   \int_{\torus} \varphi  \,\dd \omega_{(\cdot)}^\eps\Big|_{t_1}^{t_2}
    =
    - &\int_{t_1}^{t_2} \int_{\torus} \varphi \, \frac1\eps \big|\eps\Delta \uu_\eps{-} \partial_{\uu}F_\eps(\uu_\eps)\big|^2 \,\dd x \,\dd t
    \\
    \notag
    -& \int_{t_1}^{t_2}\int_{\torus} \sum_{i=1}^n (I{-}\nu^\eps_i\otimes \nu^\eps_i) \cdotdot \nabla^2 \varphi \, \eps |\nabla u_{\eps,i}|^2\, \dd x \,\dd t
    \\
    +& \int_{t_1}^{t_2}\int_{\torus} \partial_t \varphi \, \dd \omega^\eps_t\,\dd t
    + \int_{t_1}^{t_2}\int_{\torus} \Delta\varphi  \,\dd \xi^\eps_t  \,\dd t.
    \label{eq:epsBrakkeintegrated}
  \end{align}
  The left-hand side as well as the penultimate right-hand side term converge since $\omega^\eps_t\to \omega_t$ for all $t\geq0$, while the last right-hand side term vanishes in the limit $\eps\to0$ by Theorem~\ref{thm:discrepancy measure full}.
  The second right-hand side term is, due to Definition~\ref{def:varifold}, precisely 
  $-\int_{t_1}^{t_2}\delta \mu^\eps_t(\nabla \varphi)\,\dd t$.
  By Proposition~\ref{prop:measurability}, and since for a.e.\ $t\in (0,\infty)$, we have $\delta \mu^\eps_t(\nabla \varphi) \geq - \|\nabla^2\varphi\|_\infty \int_{\torus} \eps |\nabla \uu_\eps|^2 \,\dd x \geq -\|\nabla^2 \varphi\|_\infty 2 E_0$, we may apply Fatou's lemma to obtain
  \begin{align*}
    \liminf_{\eps\to0}\int_{t_1}^{t_2}\delta \mu^\eps_t(\nabla \varphi)\,\dd t
    \geq \int_{t_1}^{t_2} \liminf_{\eps\to0} \delta \mu^\eps_t(\nabla \varphi)\,\dd t.
  \end{align*}
  Fix $t>0$. Let $\eps=\eps(t)\to0$ denote a subsequence such that 
  \[
    \liminf_{\eps\to0} \delta \mu^\eps_t(\nabla \varphi) = \lim_{\eps(t)\to0} \delta \mu_t(\nabla \varphi).
  \] 
  By Proposition~\ref{prop:compactness of varifolds}, for a.e.\ $t>0$, we find a further subsequence $\eps'=\eps'(t)\to0$ such that $ \mu^{\eps'(t)}_t\stackrel{\ast}{\rightharpoonup} \mu_t$ and hence also $\delta\mu^{\eps'(t)}_t\stackrel{\ast}{\rightharpoonup} \delta \mu_t$. 
  Therefore,
  \begin{align}\label{eq:Brakke_transport_limit}
    \liminf_{\eps\to0}\int_{t_1}^{t_2}\delta \mu^\eps_t(\nabla \varphi) 
    \geq   \int_{t_1}^{t_2}\lim_{\eps'(t)\to0}\delta \mu^{\eps'(t)}_t(\nabla \varphi) = \int_{t_1}^{t_2} \delta \mu_t(\nabla \varphi)\,\dd t.
  \end{align}
  By Proposition~\ref{prop:7epsilon}, this is the desired transport term.
  Moreover, localizing the argument from the proof of Proposition~\ref{prop:compactness of varifolds} and applying Fatou's lemma in $t$ yields
  \begin{align}\label{eq:Brakke_shrinkage_limit}
    \liminf_{\eps\to0}\int_{t_1}^{t_2} \int_{\torus} \varphi \, \frac1\eps \big|\eps\Delta \uu_\eps{-} \partial_{\uu}F_\eps(\uu_\eps)\big|^2 \,\dd x \,\dd t 
    \geq \int_{t_1}^{t_2}\int_{\torus} \varphi |\vec H_t|^2 \,\dd \omega_t.
  \end{align}
  Applying~\eqref{eq:Brakke_transport_limit} and~\eqref{eq:Brakke_shrinkage_limit} to the remaining terms in~\eqref{eq:epsBrakkeintegrated} yields Brakke's inequality~\eqref{eq:Brakke} for the limit $\omega_t$.
\end{proof}

\appendix
\section{Gamma-convergence of the energies}
\label{sec:Gamma-convergence}
  Although the $\Gamma$-convergence is quite natural in view of the general work of Baldo~\cite{Baldo}, our setting does not quite fall into his framework. 
  Therefore, we state the main results and give short ideas of the simple proofs in this section.

  To formulate the $\Gamma$-convergence result rigorously, we first extend $E_\eps$ and $E$ to the common space $L^1$:
  \begin{equation}\label{eq:Eeps_rigorously}
    E_\eps(\uu):=
    \begin{cases}
        \int_{\torus} \big(\frac{\eps}2 |\nabla \uu|^2 + F_\eps(\uu) \big)\, \dd x, 
       & \uu\in H^1(\torus;\R^n),\\
      +\infty, 
      & \text{else},
    \end{cases}
  \end{equation}
  and
  \begin{align}\label{eq:E_rigorously}
  E(\uu) := 
  \begin{cases}
    \sigma \sum_{i=1}^n \H^{d-1}(\partial^\ast \Omega_i),& \uu = \sum_{i=1}^n \chi_{\Omega_i} \ee_i, \, (\Omega_i)_{i=1}^N\text{ a Caccioppoli partition of } \torus,\\
    +\infty, & \text{else,}
  \end{cases}
\end{align}
  for $\uu\in L^1(\torus;\R^n)$.
  \begin{proposition}[$\Gamma$-convergence]\label{prop:Gamma_convergence}
    The energies~$E_\eps$ defined in~\eqref{eq:Eeps_rigorously} $\Gamma$-converge to $E$ given by~\eqref{eq:E_rigorously} with respect to the $L^1$-topology as $\eps\to0$.
    Precisely,
    \begin{enumerate}[(i)]
      \item For any sequence $\uu_\eps \to \uu$ in $L^1(\torus;\R^n)$ as $\eps\to0$, it holds
      \[
        \liminf_{\eps\to0} E_\eps(\uu_\eps)\geq E(\uu).
      \] 
      \item For every $\uu \in L^1(\torus;\R^n)$ there exists a sequence $\uu_\eps \to \uu$ in $L^1(\torus;\R^n)$ such that
      \[
        \limsup_{\eps\to0} E_\eps(\uu_\eps)\leq E(\uu).
      \] 
    \end{enumerate}
  \end{proposition}

  \begin{proposition}[$\Gamma$-compactness]\label{prop:Gamma_compactness}
    The energies~$E_\eps$ defined in~\eqref{eq:Eeps} are $\Gamma$-compact in the $L^1$-topology.
    More precisely, any sequence $(\uu_\eps)_\eps \subset L^1(\torus;\R^n)$ with $\sup_{\eps\in (0,1]}E_\eps(\uu_\eps)<+\infty$ is pre-compact in $L^1(\torus;\R^n)$.
  \end{proposition}

  \begin{proof}[Idea of proof for Proposition~\ref{prop:Gamma_convergence}]
    \step{Step 1: Lower bound.} If $\uu$ is such that $E(\uu) <+\infty$, then $\uu=\sum_{i=1}^n \chi_{\Omega_i} \ee_i$ for a Caccioppoli partition $(\Omega_i)_{i=1,\ldots,n}$. 
    In particular, each $\Omega_i$ is a set of finite perimeter. 
    By dropping the non-negative coupling term between the components and then applying the Modica-Mortola~\cite{ModicaMortola}/Bogomolnyi~\cite{Bogomolnyi} trick, we obtain
    \begin{align*}
      E_\eps(\uu_\eps) 
      \geq \sum_{i=1}^n \int_{\torus} \Big( \frac\eps2 |\nabla u_{\eps,i}|^2 + \frac1\eps W(u_{\eps,i})\Big)\,\dd x
      &\geq \sum_{i=1}^n \int_{\torus} 2\sqrt{W(u_{\eps,i})} |\nabla u_{\eps,i}| \,\dd x
      \\&= \sum_{i=1}^n \int_{\torus}  |\nabla (\phi \circ u_{\eps,i})| \,\dd x.
    \end{align*}
    By the lower semi-continuity of the $BV$-seminorm w.r.t.\ $L^1$-convergence and since $\phi(u_i) = \sigma \chi_{\Omega_i}$, we get
    \[
      \liminf_{\eps\to0}E_\eps(\uu_\eps) 
      \geq  \sum_{i=1}^n \int_{\torus}  |\nabla (\phi \circ u_{i})| \,\dd x
      = \sigma \sum_{i=1}^n \H^{d-1}(\partial^\ast \Omega_i) = E(\uu).
    \]
    In the remaining case of $E(\uu)=+\infty$ it is easy to see that $\lim_{\eps\to0} E_\eps(\uu_\eps) =+\infty$.
    The only additional case compared to the literature is if the sum constraint $\sum_{i=1}^n u_i=1$ is violated, but the constraint $\uu \in \{\ee_1,\ldots,\ee_n\}$ is respected. 
    Then the coupling term $\int_{\torus}\frac1{\eps^\alpha}\big(\sigma - \sum_{i=1}^n \phi(u_{\eps,i})\big)^2 \, \dd x$ diverges.

    \step{Step 2: Upper bound.} 
    Note that for any $\uu\colon \torus\to \R^n$ and for $\eps,\delta>0$ such that $\eps^{1-\alpha} <\delta$ we have
    \begin{equation}\label{eq:Gdelta}
      F_\eps(\uu) \leq \frac1\eps G_\delta(\uu) \quad \text{with} \quad 
      G_\delta(\uu) := \sum_{i=1}^n W(u_i) + \frac\delta{2} \Big( \sigma - \sum_{i=1}^n \phi (u_i)\Big)^2.
    \end{equation}
    Then $G_\delta$ is a regular multi-well potential with zeros $\{\ee_1,\ldots, \ee_n\}$, and we define 
    \begin{equation}\label{eq:Eeps_rigorously_delta}
      \tilde E_\eps^{(\delta)}(\uu):=
      \begin{cases}
       \int_{\torus} \big(\frac{\eps}2 |\nabla \uu|^2 + \frac1\eps G_\delta(\uu) \big)\, \dd x,
       & \uu\in H^1(\torus;\R^n),\\
      +\infty, 
      & \text{else,}
    \end{cases}
  \end{equation}
  for $\uu \in L^1(\torus;\R^n)$. Then $E_\eps(\uu) \leq E_\eps^{(\delta)}(\uu)$ for any $\uu \in L^1(\torus;\R^d)$.
  
  By Baldo's work~\cite{Baldo}, for any $\delta>0$ and for any $\uu \colon \torus \to \R^n$ there exists a recovery sequence $\uu_{\eps}^{(\delta)}$ of $\tilde E_\eps^{(\delta)}$, i.e.,
  \[
    \limsup_{\eps\to 0 } \tilde E_\eps^{(\delta)}(\uu_\eps^{(\delta)}) \leq \tilde E^{(\delta)}(\uu),
  \] 
  where the limit energy is given by\footnote{Carefully observe the different normalization in~\cite{Baldo}.
  More precisely, Baldo's energy in~\cite{Baldo} is defined as $\hat E_\eps(\uu) = \int\big( \eps |\nabla \uu|^2 + \frac1\eps \hat F(\uu)\big)\,\dd x$, and he shows $\hat E_\eps \to \hat E$ for $\hat E(\uu):= \sum_{i,j} d(\ee_i,\ee_j) \H^{d-1}(\partial^\ast \cap \partial^\ast \Omega_j)$ with $d(\ee_i,\ee_j) = \inf \int \sqrt{\hat F(\gamma)} |\gamma'|$. 
  Translated to our setting, $\tilde E_\eps^{(\delta)}=\frac12 \hat E_\eps$ with $\hat F = 2G_\delta$, Baldo's result indeed yields $\tilde E^{(\delta)}(\uu) = \frac12 \hat E (\uu)$ with $\sigma_\delta(\ee_i,\ee_j) =\frac12  d(\ee_i,\ee_j) = \inf_\gamma \sqrt{\frac12 G_\delta(\gamma)}|\gamma'|$.}
  \begin{equation}\label{eq:Edelta}
      \tilde E^{(\delta)}(\uu):=
      \begin{cases}
        \sum_{i,j=1}^n \sigma_{\delta}(\ee_i,\ee_j) \H^{d-1}(\partial^\ast \Omega_i\cap \partial^\ast \Omega_j),& \uu = \sum_{i=1}^n \chi_{\Omega_i} \ee_i, \, (\Omega_i)_{i=1}^N\text{ Cacc.\ partition},\\
        +\infty, & \text{else},
      \end{cases}
  \end{equation}
  with surface tension
  \begin{equation}
    \sigma_{\delta}(\ee_i,\ee_j)  := \inf\bigg\{ \int_0^1  \sqrt{\tfrac12G_\delta(\gamma)} |\gamma'|\,\dd s \bigg| 
    \text{$\gamma$ a $C^1$-curve connecting $\ee_i$ and $\ee_j$} \bigg\}.
  \end{equation}
  By symmetry, in our case, $\sigma_{\delta}(\ee_i,\ee_j)=:\sigma_\delta$ is independent of $i\neq j$.
  Moreover, since $G_\delta(\uu) \to G_0(\uu) = \sum_{i=1}^n W(u_i)$ as $\delta\to0$, letting $\gamma(s):= (1-s)\ee_1+s\ee_2$ denote the straight path between $\ee_1$ and $\ee_2$, we can estimate
  \begin{align*}
    \sigma_\delta 
    \leq \int_0^1  \sqrt{\tfrac12 G_\delta (\gamma(s))}|\gamma'(s)|\, \dd s
    &\leq \int_0^1  \sqrt{\tfrac12 G_0(\gamma(s))} |\gamma'(s)|\, \dd s + o(1)
  \end{align*}
  as $\delta \to0$.
  We compute for this particular path and our symmetric double-well potential $W$,
  \begin{align*}
    \int_0^1 \sqrt{\tfrac12 G_0(\gamma(s))} |\gamma'(s)|\, \dd s 
    = \int_0^1 \sqrt{\tfrac12 (W(s)+W(1-s))} \sqrt{2}\,\dd s
    = \int_0^1 \sqrt{ 2W(s)} \,\dd s
    = \sigma.
  \end{align*}
  Therefore,
  \begin{equation}
    \sigma_\delta \leq \sigma + o(1) \quad \text{as }\delta\to0,
  \end{equation}
  and hence, for any $\uu\in L^1(\torus;\R^d)$, 
  \begin{equation*}
      \limsup_{\delta\to0}\tilde E^{(\delta)}(\uu)\leq E(\uu).
  \end{equation*}
  Selecting a suitable diagonal sequence $\uu_\eps = \uu_\eps^{(\delta_\eps)}$, we obtain the tight upper bound
  \[
    \limsup_{\eps\to 0 } \tilde E_\eps^{(\delta)}(\uu_\eps) 
    \leq  \limsup_{\eps\to 0 } \tilde E_\eps^{(\delta)}(\uu_\eps^{(\delta_\eps)}) 
    \leq \tilde E(\uu),
  \]
  which concludes the proof.  
  \end{proof}

  \begin{proof}[Proof of Proposition~\ref{prop:Gamma_compactness}]
    Using Step 1 from the above proof, we see that with uniform energy bounds, the functions $\phi \circ u_{\eps,i}$ are pre-compact in $L^1$ and one can easily deduce the pre-compactness of $\uu$ in $L^1$.
  \end{proof}

\section{Construction of initial data}
\label{sec:initial_data}
In the next proposition, we construct well-prepared initial conditions. 
Loosely speaking, this is done by gluing the optimal transition profile $q^\eps$ around the interfaces. Here, 
 $q^\varepsilon (s) := \frac{e^{\sqrt{2} s/\varepsilon }}{1 + e^{\sqrt{2} s/\varepsilon}}$. Note that $q^\eps$ satisfies the equi-partition of energy $\frac{\varepsilon ((q^\varepsilon)')^2 }{2} = \frac{W(q^\varepsilon)}{\varepsilon}$, as well as the boundary conditions $q^\varepsilon (-\infty) =0$, and $q^\varepsilon (\infty) =1$, and the symmetry $q^\varepsilon (0)=\frac12$.
\begin{proposition}\label{prop7}
Assume that a family of open sets $E_1, E_2,\dots, E_n \subset \mathbb{T}^d$ satisfies the following:
\begin{enumerate}[(i)]
\item $\mathbb{T}= \cup_{i=1} ^n \overline{E_i}$ and if $i\not =j$ then $E_i \cap E_j =\emptyset$. 
\item Let $d_i $ be a singed distance function of $E_i$ with $d_i (x) > 0$ for any $x \in E_i$. For any $i=1,\dots,n$, $E_i$ is a set of finite perimeter and $|\mathcal{H}^{d-1} (\partial ^\ast E_i) -\mathcal{H}^{d-1} (\Gamma _i ^s) | \to 0$ as $s\to 0$, where $\Gamma _i ^s = \{ x \mid d_i (x)=s \}$. In addition, there exists $C>0$ such that $\mathcal{H}^{d-1} (\Gamma _i ^s) \leq C$ for a.e. $s \in \R$.
\item There exists $C>0$ such that for any $i=1,\dots,n$, $x \in \mathbb{T}^d$, and for a.e. $s \in \R$, $\mathcal{H}^{d-1} (\Gamma _i ^s \cap B_r (x)) \leq Cr^{d-1}$.
\item Set $N_i ^\varepsilon =\{ x \mid |d_i (x)| \leq \varepsilon^{\frac{1+\alpha}{2}} \}$. Then there exists $C>0$ such that $\mathcal{L}^d (N_ i ^\varepsilon )\leq C \varepsilon^\frac{1+\alpha}{2}$ for any $\varepsilon >0$ and for any $i=1,\dots,n$.
\end{enumerate}
Set $u_{0,\varepsilon,i} (x) := q^{\varepsilon} (d_i (x))$. Then 
\begin{enumerate}[(i)]
\item \label{label_initial1} $0\leq u_{0,\varepsilon,i} (x) \leq 1$ for any $x \in \mathbb{T}^d$ and for any $i=1,\dots,n$.
\item \label{label_initial2} $\lim_{\varepsilon \to 0} u_{0,\varepsilon,i} (x) = \chi _{E_i} (x)$ for a.e. $x \in \mathbb{T}^d$.
\item \label{label_initial3} $\varepsilon |\nabla u_{0,\varepsilon,i} (x)| \leq C$ for a.e. $x \in \mathbb{T}^d$ and for any $i=1,\dots,n$.
\item \label{label_initial4} $\displaystyle \frac{\varepsilon |\nabla u_{0,\varepsilon,i} (x)|^2 }{2} = \frac{W(u_{0,\varepsilon,i})}{\varepsilon}$ for a.e. $x \in \mathbb{T}^d$.
\item \label{label_initial5} We have
\begin{equation}
\omega_0 ^\varepsilon \stackrel{\ast}{\rightharpoonup} \sigma \sum_{i=1} ^n \mathcal{H}^{d-1} \llcorner \partial ^\ast E_i.
\end{equation}
\item \label{label_initial6} We have $D_0 <\infty$, where $D_0$ is defined in Definition \ref{def:well-prepared}.
\end{enumerate}
\end{proposition}
\begin{proof}
First, (\ref{label_initial1}) and (\ref{label_initial2}) are clear from the definition of $u_{0,\varepsilon,i}$.
Since $|\nabla d_i (x)| = 1$ for a.e.~$x \in \mathbb{T}^d$ and $\frac{\varepsilon ((q^\varepsilon)')^2 }{2} = \frac{W(q^\varepsilon)}{\varepsilon}$, we have (\ref{label_initial3}) and (\ref{label_initial4}). 
We prove (\ref{label_initial5}) in the following. For any $\varphi \in C (\mathbb{T}^d)$, we compute
\begin{equation}\label{label_initial_eq1}
\begin{split}
& \int_{\torus} \varphi \left(\frac{\varepsilon |\nabla u_{0,\varepsilon,i} (x)|^2 }{2} + \frac{W(u_{0,\varepsilon,i})}{\varepsilon} \right) \, \dd x 
=\int_{\torus} \varphi \left( |\nabla u_{0,\varepsilon,i} (x)| \sqrt{2W(u_{0,\varepsilon,i})} \right) \, \dd x \\
& \, = \int_0 ^1 \sqrt{2W(s)} \int_{\{ x \mid u_{0,\varepsilon,i} (x) =s \} } \varphi \, \dd \mathcal{H}^{d-1} \dd s
= \int_0 ^1 \sqrt{2W(s)} \int_{ \Gamma_i ^{(q^{\varepsilon })^{-1} (s)} } \varphi \, \dd \mathcal{H}^{d-1} \dd s \\
& \to \sigma \int _{ \partial^\ast E_i } \varphi \, \dd \mathcal{H}^{d-1},
\end{split}
\end{equation}
where we used the co-area formula, $|\mathcal{H}^{d-1} (\partial ^\ast E_i) -\mathcal{H}^{d-1} (\Gamma _i ^s) | \to 0$ and $\mathcal{H}^{d-1} (\Gamma _i ^s) \leq C$. By \eqref{label_initial_eq1}, we only need to show that the coupling-term in the energy vanishes:
\begin{equation}\label{label_initial_eq2}
\int_{\torus}\frac{1}{2\varepsilon ^\alpha} \left( \sigma - \sum_{i=1} ^n \phi (u_{0,\varepsilon,i}) \right)^2 \, \dd x \longrightarrow 0.
\end{equation}
On the one hand we estimate 
\begin{equation*}
\begin{split}
\int_{N_i ^\varepsilon}\frac{1}{2\varepsilon ^\alpha} \left( \sigma - \sum_{i=1} ^n \phi (u_{0,\varepsilon,i}) \right)^2 \, \dd x \leq C \varepsilon ^{\frac{1+\alpha}{2}-\alpha}
= C \varepsilon ^{\frac{1-\alpha}{2}}  \to 0,
\end{split}
\end{equation*}
where we used $\mathcal{L}^d (N_ i ^\varepsilon )\leq C \varepsilon^{\frac{1+\alpha}{2}}$ and $0<\alpha<1$.
On the other hand, we claim that for sufficiently small $\eps>0$,
\begin{equation}\label{eq:initial conditions in bulk}
  \Big| \sigma - \sum_{i=1} ^n \phi (u_{0,\varepsilon,i}) \Big| \leq \eps^{\alpha} \quad \text{in }\mathbb{T}^d \setminus \Big(\bigcup_{i=1} ^n N_i ^\varepsilon\Big).
\end{equation}
In fact, if $x \in \mathbb{T}^d \setminus (\bigcup_{i=1} ^n N_i ^\varepsilon)$ satisfies $x \in E_i$ and $x \not \in E_j$, then we have 
\[
  1 \geq u_{0,\varepsilon,i} (x) \geq \frac{e ^{\sqrt{2} \varepsilon ^{\frac{1+\alpha}{2}-1}}}{1+ e ^{\sqrt{2} \varepsilon ^{\frac{1+\alpha}{2}-1}}} 
  = \frac{e ^{\sqrt{2} \varepsilon ^{- \frac{1-\alpha}{2}}}}{1+ e ^{\sqrt{2} \varepsilon ^{- \frac{1-\alpha}{2}}}}
  \to 1
\]
and 
\[
  0 \leq u_{0,\varepsilon,j} (x) \leq \frac{e ^{-\sqrt{2} \varepsilon ^{\frac{1+\alpha}{2}-1}}}{1+ e ^{-\sqrt{2} \varepsilon ^{\frac{1+\alpha}{2}-1}}}
=\frac{e ^{-\sqrt{2} \varepsilon ^{-\frac{1-\alpha}{2}}}}{1+ e ^{-\sqrt{2} \varepsilon ^{ -\frac{1-\alpha}{2}}}}
\to 0,
\]
where both convergences are exponentially fast as $\eps\to 0$.
Hence, in particular \eqref{eq:initial conditions in bulk} holds true, and we obtain
\begin{equation*}
\int_{\mathbb{T}^d \setminus (\bigcup_{i=1} ^n N_i ^\varepsilon)}\frac{1}{2\varepsilon ^\alpha} \left( \sigma - \sum_{i=1} ^n \phi (u_{0,\varepsilon,i}) \right)^2 \dd x \longrightarrow 0.
\end{equation*}
Therefore we obtain \eqref{label_initial_eq2} and thus (\ref{label_initial5}).

Finally we prove (\ref{label_initial6}). Let $x_0 \in \torus$ and $r\in (0,1]$ be arbitrary.
By a similar calculation to \eqref{label_initial_eq1}, we obtain
\begin{equation}\label{label_initial_eq3}
\begin{split}
 \int_{B_r (x_0)} \left(\frac{\varepsilon |\nabla u_{0,\varepsilon,i} (x)|^2 }{2} + \frac{W(u_{0,\varepsilon,i})}{\varepsilon} \right) \, \dd x 
  &= \int_0 ^1 \sqrt{2W(s)} \mathcal{H}^{d-1} \big(B_r (x_0) \cap \Gamma_i ^{(q^{\varepsilon })^{-1} (s)} \big) \,\dd s 
  \\
&\leq   C \sigma r^{d-1},
\end{split}
\end{equation}
where we used $\mathcal{H}^{d-1} (B_r (x_0) \cap \Gamma_i ^{(q^{\varepsilon })^{-1} (s)} ) \leq C r^{d-1}$.
Similarly, by the coarea formula, we have
\begin{equation}\label{label_initial_eq4}
\begin{split}
& \int_{N_i ^\varepsilon \cap B_r (x_0)} \frac{1}{2\varepsilon ^\alpha} \left( \sigma - \sum_{i=1} ^n \phi (u_{0,\varepsilon,i}) \right)^2 \, \dd x \\
= & \, \int_{-\varepsilon ^{\frac{1+\alpha}{2}}} ^{\varepsilon ^{\frac{1+\alpha}{2}}}
\int_{\Gamma _i ^s \cap B_r (x_0)} \frac{1}{2\varepsilon ^\alpha} \left( \sigma - \sum_{i=1} ^n \phi (u_{0,\varepsilon,i}) \right)^2 \, \dd \mathcal{H}^{d-1} \dd s 
\leq C \varepsilon ^{\frac{1-\alpha}{2}} r^{d-1} \leq Cr^{d-1},
\end{split}
\end{equation}
where we used $\alpha \in (0,\frac12)$. In addition, by \eqref{eq:initial conditions in bulk}, 
\begin{equation}\label{label_initial_eq5}
\begin{split}
& \int_{B_r (x_0) \setminus (\bigcup_{i=1} ^n N_i ^\varepsilon)} \frac{1}{2\varepsilon ^\alpha} \left( \sigma - \sum_{i=1} ^n \phi (u_{0,\varepsilon,i}) \right)^2 \, \dd x 
=\int_{B_r (x_0) \setminus (\bigcup_{i=1} ^n N_i ^\varepsilon)} \frac{\varepsilon ^\alpha}{2} \, \dd x
\\
\leq & \, \frac{\varepsilon ^\alpha}{2}
\sum_{i=1} ^n \int_{ \varepsilon ^{\frac{1+\alpha}{2}}} ^{\sqrt{d}} \mathcal{H}^{d-1} (\Gamma _i ^s \cap B_r (x_0)) \dd s 
\leq C \varepsilon ^{\alpha} r^{d-1} \leq Cr^{d-1}.
\end{split}
\end{equation}
By \eqref{label_initial_eq3}, \eqref{label_initial_eq4}, and \eqref{label_initial_eq5}, we obtain the claim.
\end{proof}

\begin{remark}\label{rem:initialconditions}
The function given in Proposition \ref{prop7} is a Lipschitz function, and by applying a suitable mollifier to it, one can obtain the well-prepared initial data in Definition \ref{def:well-prepared}. As for the condition regarding the upper bound of the density, we briefly mention the reason below.
Set $u_{0,\varepsilon,i} ^\delta := \eta_\delta \ast u_{0,\varepsilon,i}$, where $\eta_\delta \in C_ c ^\infty (B_\delta (0))$ is the standard mollifier. By the Cauchy--Schwarz inequality and $\int \eta_\delta \, \dd x =1$,
\[
\int_{B_r (x_0)} \frac{\varepsilon |\nabla u_{0,\varepsilon,i} ^\delta (x)|^2 }{2} \, \dd x
\leq \int_{B_r (x_0)} \int_{B_\delta (x)} 
\frac{\varepsilon |\nabla u_{0,\varepsilon,i} (y) |^2 }{2} \eta_\delta (x-y) \, \dd y \, \dd x =:J.
\]
If $0<r \leq \delta$, by $\eta_\delta \leq C\delta^{-d}$ we have 
\[
J\leq \frac{C}{\delta ^d}  \int_{B_r (x_0)} \int_{B_\delta (x)} 
\frac{\varepsilon |\nabla u_{0,\varepsilon,i} (y) |^2 }{2} \, \dd y\, \dd x
\leq 
\frac{C}{\delta ^d}  \int_{B_r (x_0)} D_0 \delta^{d-1} \, \dd x = C D_0 \frac{r}{\delta} r^{d-1}\leq Cr^{d-1}.
\]
Note that $D_0 <\infty$ holds for $u_{0,\varepsilon,i}$. If $0<\delta \leq r$, by Fubini's theorem and $\int \eta_\delta \, \dd x =1$,
\[
J\leq \int_{B_{r+\delta} (x_0)} 
\frac{\varepsilon |\nabla u_{0,\varepsilon,i} (y) |^2 }{2} \, \dd y 
\leq D_0 (r+\delta)^{d-1} \leq 2^{d-1} D_0 r^{d-1}.
\]
Since $W$ is Lipschitz, 
$| W(u_{0,\varepsilon,i} ^\delta) - W(u_{0,\varepsilon,i})| 
\leq L|u_{0,\varepsilon,i} ^\delta - u_{0,\varepsilon,i} |
\leq L \| \nabla u_{0,\varepsilon,i} \|_{L^\infty} \delta \leq L C\delta \varepsilon ^{-1}$. Hence,
\[
\left| 
\int_{B_r (x_0)} \frac{W(u_{0,\varepsilon,i} ^\delta)}{\varepsilon} - \frac{W(u_{0,\varepsilon,i})}{\varepsilon} \, \dd x
\right|
\leq \frac{CL \delta }{\varepsilon^2} r^{d} \leq C r^{d-1}
\]
holds if $\delta \leq \varepsilon ^2$. A similar estimate holds for the remaining potential energy.
\end{remark}

\section*{Acknowledgments}
T.L.\ acknowledges support from the Deutsche Forschungsgemeinschaft (DFG, German Research
Foundation) under Germany's Excellence Strategy EXC 2181/1 - 390900948 (the Heidelberg
STRUCTURES Excellence Cluster). 
K.T.\ is supported by JSPS KAKENHI Grant Numbers
JP23K03180, JP23H00085, JP24K00531, and JP25K00918.

\printbibliography
\end{document}